\documentclass[11pt,letterpaper, reqno, english]{amsart}
\usepackage{url,mathptmx}
\usepackage{amsmath,amstext,amsfonts,amssymb,amsbsy,mathrsfs}
\usepackage{etoolbox}
\usepackage{tikz}
\usetikzlibrary{intersections}
\pgfdeclarelayer{bg}    
\pgfsetlayers{bg,main}  
\usetikzlibrary{calc}
\usepackage{color}
\usepackage{latexsym}
\usepackage{array}
\usepackage{enumerate}
\usepackage{amsthm}
\usepackage{bbm}
\usepackage{amscd}
\usepackage{multirow}
\usepackage{babel}
\usepackage{esint}
\usepackage{hyperref}
\hypersetup{colorlinks=true, allcolors=blue}%

\newcommand{\lb}{\left\langle}
\newcommand{\rb}{\right\rangle}

\newcommand{\abs}[1]{\left|#1\right|}
\newcommand{\norm}[1]{\left\| #1 \right\|}

\newcommand{\bb}{\mathbb}

\newcommand{\Grad}{\nabla}
\newcommand{\bdy}{\partial}
\newcommand{\weakconv}{\rightharpoonup}

\renewcommand{\d}{{\rm d}}

\newcommand\xqed[1]{%
  \leavevmode\unskip\penalty9999 \hbox{}\nobreak\hfill
  \quad\hbox{#1}}
\newcommand\demo{\xqed{$\triangle$}}

\newtheorem{theorem}{Theorem}
\newtheorem{coro}[theorem]{Corollary}
\newtheorem{prop}[theorem]{Proposition}
\newtheorem{lemma}[theorem]{Lemma}
\newtheorem{claim}[theorem]{Claim}
\newtheorem{remark}[theorem]{Remark}

\numberwithin{theorem}{section}
\numberwithin{equation}{section}
\numberwithin{figure}{section}

\newif\ifdetails

\title[Conformally Invariant Extension Operators on $\bb R_+^n$]{Subcritical Approach to Conformally Invariant Extension Operators on the Upper Half Space}
\author{Mathew Gluck}
\address{Mathew Gluck, Department of Mathematics,
The University of Oklahoma, Norman, OK 73019, USA}
\email{mgluck@math.ou.edu}
\date{\today}
\begin{document}
\maketitle
\begin{abstract}
In this work we obtain sharp embedding inequalities for a family of conformally invariant integral extension operators. This family includes (among others) the classical Poisson extension operator and the extension operator with Riesz kernel. We show that the sharp constants in these inequalities are attained and classify the corresponding extremal functions. We also compute the limiting behavior at the boundary of the extensions of the extremal functions. 
\end{abstract}
\section{Introduction}
The classical Hardy-Littlewood-Sobolev (HLS) inequality \cite{HL1928, HL1930, Sobolev1938, Lieb1983} states that if $1< p, t< \infty$ and if $1< \alpha < n$ satisfy $\frac 1 p + \frac 1 t = \frac{n + \alpha}{n}$ then there exists a sharp constant $\mathcal H(n, \alpha, p)>0$ such that 
\begin{equation}
\label{eq:ClassicalHLS}
	\abs{\int_{\bb R^n}\int_{\bb R^n}\frac{f(y)g(x)}{\abs{x - y}^{n - \alpha}}\; \d y\; \d x}
	\leq
	\mathcal H(n, \alpha, p) \norm{f}_{L^p(\bb R^n)}\norm{g}_{L^t(\bb R^n)} 
\end{equation}
for all $f\in L^p(\bb R^n)$ and all $g\in L^t(\bb R^n)$. In the diagonal case $p = t = 2n/(n +\alpha)$ the extremal functions were classified and the value of the sharp constant was computed by Lieb in \cite{Lieb1983}. If in the diagonal case $\alpha = 2$ and $n\geq 3$ then sharp inequality \eqref{eq:ClassicalHLS} is dual to the classical sharp Sobolev inequality $S_n\norm{u}_{2n/(n-2)}\leq\norm{\Grad u}_2^2$ and the sharp constant for each of these inequalities gives the sharp constant for the other. The sharp Sobolev and HLS inequalities play prominent roles in many geometric problems including, for example the Yamabe problem \cite{Trudinger1968, Aubin1976, Schoen1984, LeeParker1987}. In recent years, variants and generalizations of the classical HLS inequality have been investigated, some of which also have geometric implications. For example, in \cite{FrankLieb2012} Frank and Lieb prove a sharp HLS inequality on the Heisenberg group. Another variant of \eqref{eq:ClassicalHLS} is the reversed HLS inequality of Dou and Zhu \cite{DouZhu2015} (see also \cite{NgoNguyen2017}) which applies to the case where the differential order exceeds the dimension. \\ 

Of particular interest in this paper is a family of HLS-type inequalities of the form 
\begin{equation}
\label{eq:BoundaryFamily}
	\abs{\int_{\bb R_+^n}\int_{\bdy \bb R_+^n} K(x' - y, x_n) f(y)g(x)}
	\leq
	C\norm{f}_{L^p(\bdy \bb R_+^n)}\norm{g}_{L^t(\bb R_+^n)}, 
\end{equation}
where $K$ is a kernel of the form 
\begin{equation}
\label{eq:UHSUnweightedKernel}
	K(x) 
	= 
	K_{\alpha, \beta}(x)
	= 
	\frac{x_n^\beta}
	{(\abs{x'}^2 + x_n^2)^{(n - \alpha)/2}}
\end{equation}
and $x = (x', x_n)\in \bb R^{n - 1}\times(0, \infty)$. This family of kernels includes the classical Poisson kernel $K_{0, 1}$, the Riesz kernel $K_{\alpha, 0}$ and the Poisson kernel $K_{\alpha, 1-\alpha}$ for the divergence-form operator $u\mapsto{\rm div}(x_n^\alpha \Grad u)$ on the upper half space. These three kernels are well-studied and arise in connection with many interesting problems. For example the relationship between the Poisson kernel and the isoperimetric problem for scalar-flat metrics on compact Riemannian manifolds with boundary was pointed out in \cite{HangWangYan2009}, see also \cite{GluckZhu2017, JinXiong2017}. The relationship between the kernel $K_{\alpha, 1-\alpha}$ and the fractional Laplacian was pointed out in \cite{CaffarelliSilvestre2007}. Inequalities of the form of \eqref{eq:BoundaryFamily} have been obtained for each of the kernels $K_{0, 1}$, $K_{\alpha, 0}$ and $K_{\alpha, 1 - \alpha}$, see \cite{HangWangYan2008}, \cite{DouZhu2013} and \cite{Chen2012} respectively. In fact, each of these inequalities (with their corresponding choices of $\alpha$ and $\beta$) were shown to hold for all exponents $1< p, t< \infty$ on the so-called critical hyperbola
\begin{equation}
\label{eq:CriticalHyperbola}
	\frac 1 t + \frac{n - 1}{np} = 1 + \frac{\alpha + \beta -1}{n}. 
\end{equation}
Moreover, the proofs of these inequalities all follow the classical approach of obtaining weak-type estimates then interpolating. These proofs all rely on the assumption that $\alpha + \beta \geq 1$. On the other hand, in \cite{DouGuoZhu2017} a subcritical approach was taken to prove an inequality of the form \eqref{eq:BoundaryFamily} for the kernel $K_{\alpha, 1}$ and for the conformal exponents $p = 2(n-1)/(n + \alpha -2)$ and $t = 2n/(n + \alpha + 2)$ (which satisfy \eqref{eq:CriticalHyperbola} with $\beta = 1$). Their approach also relies on the assumption that $\alpha + \beta \geq 1$. In this work a subcritical approach is used to obtain inequality \eqref{eq:BoundaryFamily} for the conformal exponents $p = 2(n-1)/(n + \alpha - 2)$, $t = 2n/(n + \alpha + 2\beta)$ and for the full range of admissible $\alpha$ and $\beta$ (i.e. in the absence of the assumption $\alpha + \beta\geq 1$). To state our results we introduce the extension operator 
\begin{equation*}
	E_{\alpha, \beta}f(x) = \int_{\bdy \bb R_+^n} \frac{x_n^\beta f(y)}{\abs{x - y}^{n - \alpha}}\; \d y
	\qquad
	\text{ for } f\in L^p(\bdy \bb R_+^n)
\end{equation*}
and the corresponding restriction operator
\begin{equation*}
	R_{\alpha, \beta}g(y) = \int_{\bb R_+^n} \frac{x_n^\beta g(x)}{\abs{x - y}^{n - \alpha}}\; \d x
	\qquad
	\text{ for } g\in L^t(\bb R_+^n). 
\end{equation*}
Note that the use of the adjective ``extension'' in this context is not meant to imply that $\lim_{x_n\to 0}E_{\alpha, \beta}f(x', x_n) = f(x', 0)$. In fact, the limiting behavior of $E_{\alpha, \beta}f$ depends crucially on $\alpha, \beta$ and is one of the topics of investigation of this work, see Theorem \ref{theorem:UHSConformalExtremalClassification} below. In the following theorem and throughout this work $B = B_1(0)$ will denote the open unit ball in $\bb R^n$. 
\begin{theorem}
\label{theorem:ConformallyInvariantSharpHalfSpaceInequality}
Let $n\geq 2$ and suppose $\alpha, \beta$ satisfy 
\begin{equation}
\label{eq:MinimalAlphaBeta}
	\beta\geq 0,\quad   0< \alpha + \beta < n - \beta
\end{equation}
and 
\begin{equation}
\label{eq:AlphaBetaSubAffine}
	\frac{n - \alpha - 2\beta}{2n} + \frac{n - \alpha}{2(n - 1)}< 1. 
\end{equation}
There exists an optimal constant $C_e(n, \alpha, \beta)$ such that for every $f\in L^{\frac{2(n -1)}{n +\alpha - 2}}(\bdy \bb R_+^n)$ the inequality 
\begin{equation}
\label{eq:SharpConformallyInvariantUHSInequality}
	\norm{E_{\alpha, \beta}f}_{L^{\frac{2n}{n - \alpha - 2\beta}}(\bb R_+^n)}
	\leq
	C_e(n, \alpha, \beta)\norm{f}_{L^{\frac{2(n - 1)}{n +\alpha -2}}(\bdy \bb R_+^n)}
\end{equation}
holds. Moreover, the value of the optimal constant is
\begin{equation}
\label{eq:OptimalExtensionConstant}
	C_e(n, \alpha, \beta) 
	= 
	(n\omega_n)^{-\frac{n + \alpha -2}{2(n - 1)}}
	\norm{\int_{\bdy B}H(\cdot, \zeta)\; \d S_\zeta }_{L^{\frac{2n}{n -\alpha - 2\beta}}(B)}, 
\end{equation}
where $H:B\times \bdy B\to \bb R$ is given by
\begin{equation}
\label{eq:BallKernel}
	H(\xi, \zeta) = H_{\alpha, \beta}(\xi, \zeta)= \left(\frac{1 - \abs{\xi}^2}{2}\right)^\beta \abs{\xi - \zeta}^{\alpha- n}. 
\end{equation}
\end{theorem}
By duality, inequality \eqref{eq:SharpConformallyInvariantUHSInequality} is equivalent to inequality \eqref{eq:BoundaryFamily} (with sharp constant) and to the following inequality for $R_{\alpha, \beta}$. 
\begin{coro}
\label{coro:ConformallyInvariantSharpHalfSpaceInequality}
Under the assumptions of Theorem \ref{theorem:ConformallyInvariantSharpHalfSpaceInequality}, for every $g\in L^{\frac{2n}{n + \alpha + 2\beta}}(\bb R_+^n)$ the inequality 
\begin{equation*}
	\norm{R_{\alpha, \beta} g}_{L^{\frac{2(n - 1)}{n - \alpha}}(\bdy \bb R_+^n)}
	\leq
	C_e(n, \alpha, \beta)\norm{g}_{L^{\frac{2n}{n + \alpha + 2\beta}}(\bb R_+^n)}
\end{equation*}
holds and the constant $C_e(n, \alpha, \beta)$ is optimal. 
\end{coro}
We'll show that the optimal constant $C_e(n, \alpha, \beta)$ in Theorem \ref{theorem:ConformallyInvariantSharpHalfSpaceInequality} is attained and we'll classify the corresponding extremal functions via the method of moving spheres. In fact, up to a constant multiple, the Euler-Lagrange equation for the nonnegative extremal functions in \eqref{eq:SharpConformallyInvariantUHSInequality} is
\begin{equation}
\label{eq:ConformallyInvariantELHalfSpace}
	f^{\frac{n - \alpha}{n + \alpha -2}}(y)
	= 
	\int_{\bb R_+^n} \frac{x_n^\beta \left(E_{\alpha, \beta}f(x)\right)^{\frac{n +\alpha + 2\beta}{n - \alpha - 2\beta}}}{\abs{x - y}^{n -\alpha}}\; \d x
	\qquad
	\text{ for } y\in \bdy \bb R_+^n. 
\end{equation}
The following theorem classifies all solutions to this equation together with the boundary behavior of the corresponding extensions. 
\begin{theorem}
\label{theorem:UHSConformalExtremalClassification}
Let $n\geq 2$ and suppose $\alpha, \beta$ satisfy \eqref{eq:MinimalAlphaBeta} and \eqref{eq:AlphaBetaSubAffine}. If $f\in L^{\frac{2(n-1)}{n + \alpha - 2}}(\bdy \bb R_+^n)$ is a nonnegative solution to \eqref{eq:ConformallyInvariantELHalfSpace} then there are constants $c\geq 0$, $d>0$ and there is $y_0\in \bdy \bb R_+^n$ such that
\begin{equation}
\label{eq:ConformalBoundaryFunctionExtremal}
	f(y) = 
	c\left(d^2 + \abs{y - y_0}^2\right)^{-\frac{n + \alpha - 2}{2}}. 
\end{equation}
Moreover for such $f$, the following limits at $\bdy \bb R_+^n$ hold for $E_{\alpha, \beta}f$: 
\begin{enumerate}[(a)]
	\item If $\alpha< 1$ then 
	\begin{equation*}
		\lim_{x_n\to 0^+}
		x_n^{1 - \alpha - \beta}E_{\alpha, \beta}f(x', x_n) 
		= 
		C_{n, \alpha}f(x'),
	\end{equation*}
	where $C_{n, \alpha} = \int_{\bb \bdy \bb R_+^n}(1 + \abs y^2)^{(\alpha - n)/2}\; \d y$. 
	\item If $\alpha = 1$ then 
	\begin{equation*}
		- \lim_{x_n\to 0^+}
		\frac{E_{1, \beta}f(x', x_n)}{x_n^\beta \log x_n}
		= 
		(n - 1)\omega_{n - 1} f(x'). 
	\end{equation*}
	\item If $1< \alpha$ then there is a positive constant $C = C(n, \alpha)$ such that 
	\begin{equation*}
		\lim_{x_n\to 0^+}x_n^{-\beta}E_{\alpha, \beta}f(x', x_n)
		= 
		Cf(x')^{\frac{n - \alpha}{n + \alpha - 2}}. 
	\end{equation*}
\end{enumerate}
\end{theorem}
This paper is organized as follows. In Section \ref{section:SubcriticalnequalityOnBall} we establish both a sharp extension inequality on the unit ball $B$ and the existence of the corresponding extremal functions for subcritical exponents. In Section \ref{section:SubcriticalExtremalClassification} we show that the extremal functions of the subcritical extension inequality on $B$ must be constant and thereby obtain an explicit expression for the sharp constant in the subcritical inequality. In Section \ref{section:ConformallyInvariantExponents}, by allowing the subcritical exponents to approach the conformal exponents, we obtain inequality \eqref{eq:SharpConformallyInvariantUHSInequality} from the subcritical inequality obtained in Section \ref{section:SubcriticalnequalityOnBall}. In Section \ref{section:ConformallyInvariantExponents} we also classify the extremal functions corresponding to \eqref{eq:SharpConformallyInvariantUHSInequality}  and compute the limiting behavior of their extensions at $\bdy \bb R_+^n$. Section \ref{section:Appendix} is an appendix containing some computations that may be useful for the reader yet detract too much from the main storyline to be included in the main body of the paper. \\

We will use the following notational conventions throughout. For $p\geq 1$, $p'$ will denote the Lebesgue-conjugate exponent, the solution to $\frac 1p + \frac 1{p'} = 1$. For $r>0$, $B_r$ will denote the open unit ball in $\bb R^n$ and we will write $B_r^{n - 1} = B_r\cap \bb \bdy \bb R_+^n$ and $B_r^+= B_r\cap \bb R_+^n$. If $\Omega$ is a subset of either $\bb R^n$ (respectively $\bdy \bb R_+^n$ or $\bb S^{n-1}$) then $\abs \Omega$ will denote either the $n$-dimensional Lebesgue measure (respectively the $n-1$-dimensional Lebesgue measure or the spherical measure) of $\Omega$. 
\section{Subcritical Inequality on the Ball}
\label{section:SubcriticalnequalityOnBall}

In this section we establish a sharp embedding inequality for an extension operator $E_B:L^p(\bdy B)\to L^s(B)$ when $p$ and $s$ are subcritical. The existence of the corresponding extremal functions is also established. Throughout, the integral kernel $H:B\times \bdy B\to \bb R$ will be as in \eqref{eq:BallKernel}. \\
\begin{prop}
\label{prop:SubcriticalBallHLSInequality}
Suppose $n\geq 2$,  $\beta\geq 0$ and $0< \alpha + \beta$. If $p>1$, $t>1$ satisfy both $1< \frac 1p + \frac 1 t$ and 
\begin{equation}
\label{eq:ptSubcritical}
	\frac 1 t + \frac{n -1}{np} < 1 + \frac{\alpha + \beta - 1}{n}
\end{equation}
then there exists a constant $C = C(n, \alpha, \beta, p, t)>0$ such that 
\begin{equation*}
	\abs{
	\int_B\int_{\bdy B}H(\xi, \zeta) f(\zeta)g(\xi)  \; \d S_\zeta\; \d \xi
	}
	\leq
	C\norm{f}_{L^p(\bdy B)}\norm{g}_{L^t(B)}
\end{equation*}
for all $f\in L^p(\bdy B)$ and all $g\in L^t(B)$. 
\end{prop}
Under the hypotheses of Proposition \ref{prop:SubcriticalBallHLSInequality} the quantity
\begin{equation}
\label{eq:BallSubcriticalHLSSharpConstant}
\begin{array}{lcl}
\multicolumn{3}{l}{
	C_*(n, \alpha, \beta, p, t)
	}
	\\
	& = & 
	\displaystyle
	\sup\left\{
	\abs{\int_B\int_{\bdy B} H(\xi, \zeta)f(\zeta) g(\xi)\; \d S_\zeta \;\d\xi}
	: \norm{f}_{L^p(\bdy B)} = 1 = \norm{g}_{L^t(B)}
	\right\}
\end{array}
\end{equation}
is well-defined and finite. We define the extension operator 
\begin{equation*}
	E_Bf(\xi) = \int_{\bdy B} H(\xi, \zeta)f(\zeta)\; \d S_\zeta
\end{equation*}
for $f:\bdy B\to \bb R$ and the restriction operator
\begin{equation*}
	R_Bg(\zeta) = \int_{B}H(\xi, \zeta) g(\xi)\; \d\xi
\end{equation*}
for $g:B\to \bb R$. We note that the use of the word extension in this context is not meant to imply that $\lim_{\xi\to \zeta\in \bdy B} E_Bf(\xi) = f(\zeta)$. Nor is the use of the word restriction meant to imply that $\lim_{\xi \to \zeta\in \bdy B} g(\xi) = R_Bg(\zeta)$. From Lebesgue duality and Proposition \ref{prop:SubcriticalBallHLSInequality} we immediately obtain the following corollary regarding the mapping properties of $E_B$ and $R_B$. 
\begin{coro}
\label{coro:SubcriticalBallExtensionRestriction}
Suppose $n\geq 2$, $\beta\geq 0$ and $0<\alpha + \beta$. 
\begin{enumerate}[(a)]
	\item If $s>0$ and if $\frac{n - 1}{n}\left(\frac 1 p - \frac{\alpha + \beta - 1}{n - 1}\right)< \frac 1 s< \frac 1 p < 1$ then 
	\begin{equation}
	\label{eq:SubcriticalBallExtension}
		\norm{E_{B} f}_{L^s(B)}\leq C_*(n, \alpha, \beta, p, s')\norm f_{L^p(\bdy B)}
	\end{equation}
	for all $f\in L^p(\bdy B)$. 
	\item If $r>0$ and if $\frac{n}{n - 1}\left(\frac 1 t - \frac{\alpha + \beta}{n}\right)< \frac 1 r< \frac 1 t < 1$ then 
	\begin{equation}
	\label{eq:SubcriticalBallRestriction}
		\norm{R_{B}g}_{L^r(\bdy B)} \leq C_*(n, \alpha, \beta, r', t)\norm g_{L^t(B)}
	\end{equation}
	for all $g\in L^t(B)$. 
\end{enumerate}
\end{coro}
Under some additional assumptions assumptions on the exponents, we can show that the best constant in \eqref{eq:SubcriticalBallExtension} is attained by some nonnegative function $f\in L^P(\bdy B)$. Specifically, the following holds.\\ 
\begin{prop}
\label{prop:SubcriticalBallExtremalAttained}
Let $n\geq 2$ and suppose $\alpha, \beta$ satisfy \eqref{eq:MinimalAlphaBeta} and \eqref{eq:AlphaBetaSubAffine}. If $p, s\in \bb R$ satisfy 
\begin{equation}
\label{eq:psDoublySubcritical}
	\frac{n - \alpha - 2\beta}{2n} < \frac 1 s < \frac 1 p < \frac{n +\alpha - 2}{2(n - 1)}
\end{equation}
then there is $0\leq f\in L^p(\bdy B)$ for which both $\norm{f}_{L^p(\bdy B)} = 1$ and $\norm{E_B f}_{L^s(B)} = C_*(n, \alpha, \beta, p, s')$. 
\end{prop}
The remainder of this section is devoted to the proofs of Propositions \ref{prop:SubcriticalBallHLSInequality} and \ref{prop:SubcriticalBallExtremalAttained}. The proof of Proposition \ref{prop:SubcriticalBallHLSInequality} is based on the following lemma.
\begin{lemma}
\label{lemma:ForYoungsInequality}
Let $n\geq 2$ and suppose $\beta\geq 0$ and $1 - n< \alpha + \beta$. If $p>1$ and $t>1$ satisfy \eqref{eq:ptSubcritical} then there exists $0< a< 1$ depending on $n, \alpha, \beta, p, t$ such that both 
\begin{equation}
\label{eq:BallKernelBoundarySup}
	\sup_{\zeta\in \bdy B} \int_B H(\xi, \zeta)^{at'}\; \d \xi < \infty
\end{equation}
and
\begin{equation}
\label{eq:BallKernelInteriorSup}
	\sup_{\xi \in B}\int_{\bdy B} H(\xi, \zeta)^{(1 - a)p'}\; \d S_\zeta < \infty.
\end{equation}
\end{lemma}
\begin{proof}[Proof ]
First observe that for all $\xi \in B$ and $\zeta\in \bdy B$,
\begin{equation*}
	1 - \abs \xi^2 
	= 
	(\zeta - \xi)(\zeta + \xi)
	\leq
	2\abs{\zeta - \xi}. 
\end{equation*}
Since $\beta\geq 0$ this gives
\begin{equation}
\label{eq:BallKernelUpperBound}
	H(\xi, \zeta)
	\leq
	\abs{\xi - \zeta}^{\alpha + \beta - n}
	\qquad
	\text{ for all } \xi\in B, \; \zeta\in \bdy B. 
\end{equation}
If $\alpha + \beta \geq n$ then $H$ is bounded on $B\times \bdy B$ and the assertion of the lemma follows immediately. Assume henceforth that $\alpha + \beta < n$. Since both $p>1$ and $t>1$, assumption \eqref{eq:ptSubcritical} guarantees the existence of $a\in (0, 1)$ satisfying 
\begin{equation*}
	1 - \frac{n - 1}{(n - \alpha - \beta)p'}
	< 
	a
	<
	\frac{n}{(n - \alpha -\beta)t'}. 
\end{equation*}
Fix any such $a$. For any $\zeta\in \bdy B$ inequality \eqref{eq:BallKernelUpperBound} gives
\begin{eqnarray*}
	\int_BH(\xi, \zeta)^{at'}\; \d \xi
	&\leq & 
	\int_B \abs{\xi - \zeta}^{(\alpha + \beta - n)at'}\; \d \xi
	\\
	&\leq & 
	\int_{B(\zeta, 2)}\abs{\xi - \zeta}^{(\alpha + \beta - n)a t'}\; \d \xi
	\\
	&\leq & 
	C(n, \alpha, \beta, t, a), 
\end{eqnarray*}
the final estimate holding as $(n-\alpha - \beta)at' < n$. Estimate \eqref{eq:BallKernelBoundarySup} is established. \\

To show that \eqref{eq:BallKernelInteriorSup} holds define $\pi(\xi) = \xi/\abs \xi$ for $\xi\in B\setminus\{0\}$ and $\pi(0) = 0$. For any $\zeta\in \bdy B$ and $\xi\in B$ there holds
\begin{equation*}
	\abs{\zeta - \pi(\xi)}
	\leq 
	\abs{\zeta - \xi} + \abs{\xi - \pi(\xi)}
	\leq
	2\abs{\zeta - \xi}. 
\end{equation*}
Consequently for all $\xi\in B$
\begin{eqnarray*}
	\int_{\bdy B} H(\xi, \zeta)^{(1 - a)p'}\; \d S_\zeta
	&\leq & 
	\int_{\bdy B}\abs{\xi - \zeta}^{(\alpha + \beta - n)(1 - a)p'}\; \d S_\zeta
	\\
	&\leq & 
	C\int_{\bdy B}\abs{\zeta - \pi(\xi)}^{(\alpha + \beta - n)(1 - a)p'}\; \d S_\zeta
	\\
	&\leq & 
	C 
\end{eqnarray*}
for some positive constant $C= C(n, \alpha, \beta, p, a)$, the final estimate holding as $(n - \alpha - \beta)(1 - a)p' < n - 1$. Estimate \eqref{eq:BallKernelInteriorSup} is established. 
\end{proof}
\begin{proof}[Proof of Proposition \ref{prop:SubcriticalBallHLSInequality}]
We assume with no loss of generality that $f$ and $g$ are nonnegative. Since $p>1$ and $t>1$ satisfy $1< \frac 1 p + \frac 1 t$ there is $q>1$ for which $\frac 1 p + \frac 1 t + \frac 1 q = 2$. In particular the conjugate exponents are positive and satisfy $\frac 1{p'} + \frac 1{t'} + \frac 1{q'} = 1$. By Lemma \ref{lemma:ForYoungsInequality} there is $a\in (0, 1)$ depending on $n, \alpha, \beta, p$ and $t$ for which both \eqref{eq:BallKernelBoundarySup} and \eqref{eq:BallKernelInteriorSup} hold. Fix any such $a$ and define
\begin{eqnarray*}
	\gamma_1(\xi, \zeta) & = & f(\zeta)^{\frac{p}{t'}}H^a(\xi, \zeta)\\
	\gamma_2(\xi, \zeta) & = & g(\xi)^{\frac t{p'}} H^{1- a}(\xi, \zeta)\\
	\gamma_3(\xi, \zeta) & = & f(\zeta)^{\frac{p}{q'}} g(\xi)^{\frac{t}{q'}}. \\
\end{eqnarray*}
By H\"older's inequality we have
\begin{equation}
\label{eq:YoungArgumentApplyHolder}
\begin{array}{lcl}
	\multicolumn{3}{l}{
	\displaystyle
	\int_B\int_{\bdy B} f(\zeta)g(\xi) H(\xi, \zeta)\; \d S_\zeta\; \d \xi
	}
	\\
	&= & 
	\displaystyle
	\int_B\int_{\bdy B} \gamma_1(\xi, \zeta)\gamma_2(\xi, \zeta)\gamma_3(\xi, \zeta)\; \d S_\zeta\; \d \xi
	\\
	&\leq & 
	\displaystyle
	\norm{\gamma_1}_{L^{t'}(B\times\bdy B)}\norm{\gamma_2}_{L^{p'}(B\times\bdy B)}
	\norm{\gamma_3}_{L^{q'}(B\times\bdy B)}
	\\
	&\leq & 
	\displaystyle
	\norm{\gamma_1}_{L^{t'}(B\times\bdy B)}\norm{\gamma_2}_{L^{p'}(B\times\bdy B)}
	\norm{f}_{L^p(\bdy B)}^{\frac{p}{q'}}\norm{g}_{L^t(B)}^{\frac{t}{q'}}. 
\end{array}
\end{equation}
By the choice of $a$ there is a constant $C = C(n, \alpha, \beta, p, t)>0$ such that both
\begin{equation*}
	\norm{\gamma_1}_{L^{t'}(B\times \bdy B)}^{t'}
	=
	\int_{\bdy B}f(\zeta)^{p}\int_B H^{at'}(\xi, \zeta)\; \d \xi\; \d S_\zeta
	\leq 
	C\norm{f}_{L^p(\bdy B)}^p
\end{equation*}
and
\begin{equation*}
	\norm{\gamma_2}_{L^{p'}(B\times \bdy B)}^{p'}
	=
	\int_B g(\xi)^t\int_{\bdy B}H^{(1 - a)p'}(\xi, \zeta)\; \d S_\zeta\; \d \xi
	\leq
	C\norm{g}_{L^t(B)}^t. 
\end{equation*}
The conclusion of Proposition \ref{prop:SubcriticalBallHLSInequality} follows by using these estimates in \eqref{eq:YoungArgumentApplyHolder}.
\end{proof}
The proof of Proposition \ref{prop:SubcriticalBallExtremalAttained} is based on the following compactness lemma for $E_B$. \\

\begin{lemma}
\label{lemma:BallExtensionDoublySubcriticalCompact}
Suppose $n\geq 2$ and let $\beta\geq 0$ and $\alpha$ satisfy both \eqref{eq:MinimalAlphaBeta} and \eqref{eq:AlphaBetaSubAffine}. If $p, s\in \bb R$ satisfy \eqref{eq:psDoublySubcritical} then the extension operator $E_B:L^p(\bdy B)\to L^s(B)$ is compact. 
\end{lemma}
\begin{proof}[Proof of Lemma \ref{lemma:BallExtensionDoublySubcriticalCompact}]
Let $(f_i)_{i = 1}^\infty$ be bounded in $L^p(\bdy B)$. By Corollary \ref{coro:SubcriticalBallExtensionRestriction} (a), $\left(E_Bf_i\right)_{i = 1}^\infty$ is bounded in $L^s(B)$. By reflexivity of $L^p(\bdy B)$ and $L^s(B)$ there are $f\in L^p(\bdy B)$, $F\in L^s(B)$ and a subsequence of $f_i$ (still denoted $f_i$) along which both $f_i\rightharpoonup f$ weakly in $L^p(\bdy B)$ and $E_Bf_i\rightharpoonup F$ weakly in $L^s(B)$. Moreover $F = E_B f$. Indeed, for any $g\in L^{s'}(B)$ Corollary \ref{coro:SubcriticalBallExtensionRestriction} (b) guarantees that $R_Bg\in L^{p'}(\bdy B)$. For any such $g$
\begin{equation*}
	\lb E_B f_i, g\rb
	=
	\lb f_i, R_B g\rb
	\to
	\lb f, R_Bg\rb
	= 
	\lb E_B f, g\rb. 
\end{equation*}
It remains to show that there is a subsequence of $(f_i)$ along which $E_Bf_i\to E_B f$ in $L^s(B)$. For any $0< \epsilon< 1$ there holds
\begin{equation}
\label{eq:SubcriticalExtensionDomainPartition}
	\norm{E_B(f_i - f)}_{L^s(B)}^s
	= 
	\norm{E_B(f_i - f)}_{L^s(B_{1 - \epsilon})}^s
	+ 
	\norm{E_B(f_i - f)}_{L^s(B\setminus B_{1 - \epsilon})}^s. 
\end{equation}

To estimate the first term on the right-hand side of \eqref{eq:SubcriticalExtensionDomainPartition} note that $0< H(\xi, \zeta)\leq \epsilon^{\alpha + \beta -n}$ for all $\xi\in B_{1 - \epsilon}$ and all $\zeta\in \bdy B$. In particular, for all $\xi\in B_{1 - \epsilon}$ we have $\zeta\mapsto H(\xi, \zeta)\in L^{p'}(\bdy B)$. For any such $\xi$, the $L^p(\bdy B)$-weak convergence $f_i\weakconv f$ guarantees the pointwise convergence
\begin{equation*}
	E_B(f_i - f)(\xi)
	= 
	\int_{\bdy B}H(\xi, \zeta)(f_i (\zeta) - f(\zeta))\; \d S_\zeta
	\to 
	0
\end{equation*}
and H\"older's inequality gives
\begin{equation*}
	\abs{E_B(f_i - f)(\xi)}
	\leq
	\epsilon^{\alpha + \beta - n}\int_{\bdy B}\abs{f_i - f}(\zeta)\; \d S_\zeta
	\leq
	\epsilon^{\alpha + \beta - n}\abs{\bdy B}^{\frac 1{p'}}\sup_i\norm{f_i - f}_{L^p(\bdy B)}.
\end{equation*}
The Bounded Convergence Theorem now guarantees that $\norm{E_B(f_i - f)}_{L^s(B_{1 - \epsilon})}^s\to 0$ as $i\to\infty$. \\

Consider next the second term on the right-hand side of \eqref{eq:SubcriticalExtensionDomainPartition}. Since $p> \frac{2(n - 1)}{n + \alpha - 2}$ applying H\"older's inequality then applying Corollary \ref{coro:SubcriticalBallExtensionRestriction} (a) with exponents $p$ and $2n/(n - \alpha - 2\beta)$ gives
\begin{eqnarray*}
	\norm{E_B(f_i - f)}_{L^s(B\setminus B_{1 - \epsilon})}^s
	& \leq & 
	\abs{B\setminus B_{1 - \epsilon}}^{1 - \frac{s(n - \alpha - 2\beta)}{2n}}
	\norm{E_B(f_i - f)}_{L^\frac{2n}{n - \alpha - 2\beta}(B\setminus B_{1 - \epsilon})}^s
	\\
	&\leq & 
	C \epsilon^{1 - \frac{s(n - \alpha - 2\beta)}{2n}}\norm{f_i - f}_{L^p(\bdy B)}^{s}
	\\
	&\leq & 
	C\epsilon^{1 - \frac{s(n - \alpha - 2\beta)}{2n}} \sup_i \norm{f_i - f}_{L^p(\bdy B)}^{s}
\end{eqnarray*}
for some constant $C = C(n, \alpha, \beta, p, s)>0$. Finally, returning to \eqref{eq:SubcriticalExtensionDomainPartition} we have a subsequence $f_i$ and a positive constant $C(n, \alpha, \beta, p, s)$ such that for all $0< \epsilon< 1$, 
\begin{equation*}
	\norm{E_B(f_i - f)}_{L^s(B)}^s 
	\leq 
	C\epsilon^{1 - \frac{s(n - \alpha - 2\beta)}{2n}} + \circ(1). 
\end{equation*}
as $i\to \infty$. The $L^s(B)$-convergence $E_B f_i\to E_B f$ follows.
\end{proof}
\begin{proof}[Proof of Proposition \ref{prop:SubcriticalBallExtremalAttained}]
Let $(f_i)_{i = 1}^\infty\subset L^p(\bdy B)$ be a sequence of nonnegative functions satisfying both $\norm{f_i}_{L^p(\bdy B)} = 1$ for all $i$ and $\norm{E_Bf_i}_{L^s(\bdy B)}\to C_*(n, \alpha, \beta, p, s')$. By Lemma \ref{lemma:BallExtensionDoublySubcriticalCompact} there is $f\in L^p(\bdy B)$ and a subsequence of $(f_i)$ (still denoted $(f_i)$) along which $E_Bf_i\to E_B f$ in $L^s(B)$. In particular, $\norm{E_B f}_{L^s(B)} = C_*(n, \alpha, \beta, p, s')$ so $f$ does not vanish identically. Since $(f_i)_{i = 1}^\infty$ is bounded in $L^p(\bdy B)$ we have $f_i\weakconv f$ weakly in $L^p(\bdy B)$. Testing this weak convergence against $f^{\frac{p}{p'}}\in L^{p'}(\bdy B)$ and using H\"older's inequality gives
\begin{equation*}
	\norm{f}_{L^p(\bdy B)}^{\frac{p}{p'}}
	=
	\norm{f_i}_{L^p(\bdy B)}\norm{f}_{L^p(\bdy B)}^{\frac{p}{p'}}
	\geq
	\int_{\bdy B}f_i(\zeta) f(\zeta)^{\frac{p}{p'}}\; \d S_\zeta
	\to 
	\int_{\bdy B} f(\zeta)^p\; \d S_\zeta
\end{equation*}
so that $\norm{f}_{L^p(\bdy B)}\leq 1$. Finally,
\begin{equation*}
	C_*(n, \alpha, \beta, p, s')
	=
	\norm{E_Bf}_{L^s(B)}
	\leq
	\frac{\norm{E_Bf}_{L^s(B)}}{\norm{f}_{L^p(\bdy B)}}
	\leq
	C_*(n, \alpha, \beta, p, s')
\end{equation*}
so that $\norm{f}_{L^p(\bdy B)} = 1$. 
\end{proof} 
%
\section{Classification of Extremal Functions in the Subcritical Case}
\label{section:SubcriticalExtremalClassification}
%
If $f\in L^p(\bdy B)$ is a nonnegative extremal function for inequality \eqref{eq:SubcriticalBallExtension} then up to a positive constant multiple, $f$ is a weak solution to the Euler-Lagrange equation
\begin{equation}
\label{eq:BallExtremalEulerLagrange}
	f^{p - 1}(\zeta) = \int_B H(\xi, \zeta)(E_Bf(\xi))^{s - 1}\; \d \xi
\end{equation}
for $\zeta\in \bdy B$. In this section we will show that any such function is constant and we will compute the value of the best constant $C_*$ in \eqref{eq:BallSubcriticalHLSSharpConstant}. Specifically, we have the following theorem. 
\begin{theorem}
\label{theorem:ConstantExtremalsBestEmbeddingConstant}
Let $n\geq2$ and suppose $\alpha$ and $\beta$ satisfy \eqref{eq:MinimalAlphaBeta}. If $p$ and $s$ satisfy \eqref{eq:psDoublySubcritical} and if $f\in L^p(\bdy B)$ is a nonnegative solution to \eqref{eq:BallExtremalEulerLagrange} then $f$ is constant. Consequently the optimal constant $C_*(n, \alpha, \beta, p, t)$ in \eqref{eq:BallSubcriticalHLSSharpConstant} is given by 
\begin{eqnarray*}
	C_*(n, \alpha, \beta, p, t)	
	& = & 
	(n\omega_n)^{-\frac 1p}
	\left(
	\int_B \left(
	\int_{\bdy B} H(\xi, \zeta)\; \d S_\zeta
	\right)^{t'}\; \d \xi
	\right)^{\frac{1}{t'}}. 
\end{eqnarray*}
\end{theorem}
It is routine to show that the constant $C_*$ as given in Theorem \ref{theorem:ConstantExtremalsBestEmbeddingConstant} is well-defined. For convenience a proof of this fact is given in Lemma \ref{lemma:MappingPropertiesBallExtension} of the appendix. In fact, Lemma \ref{lemma:MappingPropertiesBallExtension} proves a stronger result which also guarantees that the constant $C_e(n, \alpha, \beta)$ given in \eqref{eq:OptimalExtensionConstant} is well-defined. The major step in the proof of Theorem \ref{theorem:ConstantExtremalsBestEmbeddingConstant} is in proving symmetry about the $x_n$-axis of solutions to a corresponding system of equations on the upper half space, see Proposition \ref{prop:RadialSymmetryOfExtremalSystemSolutions} below. We start by establishing some notation and listing some elementary facts. Define $T:\overline B \to \overline{\bb R_+^n}\cup\{\infty\}$ by 
\begin{equation}
\label{eq:BallToUHSConformalMap}
	T(\xi)
	= 
	- 2e_n + \frac{4(\xi + e_n)}{\abs{\xi + e_n}^2}, 
\end{equation}
where $e_n = (0', 1) \in \bb R^n$. Evidently, $T^{-1}: \overline{\bb R_+^n} \cup\{\infty\}\to \overline B$ is given by 
\begin{equation*}
	T^{-1}(x) = - e_n + \frac{4(x + 2e_n)}{\abs{x + 2e_n}^2}
	\qquad
	 x\in \bb R_+^n \cup\{\infty\}
\end{equation*}
and $T(\bdy B) = \bdy \bb R_+^n\cup\{\infty\}$. By directly computing one can verify that 
\begin{equation}
\label{eq:DistanceBetweenPreimages}
	\abs{T^{-1}(x) - T^{-1}(y)}
	= 
	w(x)w(y)\abs{x - y}
\end{equation}
for all $x, y\in \overline{\bb R_+^n}$ and that 
\begin{equation}
\label{eq:DistanceToBoundaryUnderPreimage}
	1 - \abs{T^{-1}(x)}^2 = 2w(x)^2 x_n, 
\end{equation}
where $w$ is used to denote the weight function 
\begin{equation}
\label{eq:SubcriticalWeight}
	w(x) = \frac{2}{\abs{x + 2e_n}}
	\qquad
	\text{ for } x\in \overline{\bb R_+^n}. 
\end{equation}
For notational convenience we define
\begin{equation}
\label{eq:SubcriticalWeightExponents}
	\sigma = n + \alpha - 2 - \frac{2(n - 1)}{p}
	\quad \text{ and } \quad 
	\tau = n + \alpha + 2\beta - \frac{2n}{t}
\end{equation}
and the new exponents 
\begin{equation*}
	\theta = p' - 1 = \frac 1 {p -1} 
	\qquad
	\text{ and }
	\qquad
	\kappa = t' - 1. 
\end{equation*}

The system of equations on the upper half space corresponding to equation \eqref{eq:BallExtremalEulerLagrange} is given in the following. 
\begin{lemma}
\label{lemma:SystemTRotationTransformation}
If $f\in L^p(\bdy B)$ is a nonnegative solution to \eqref{eq:BallExtremalEulerLagrange} with $s = t'$ then for any rotation $\rho:\bb R^n\to \bb R^n$ about the origin the functions 
\begin{equation}
\label{eq:RotatedBallToUHSTransformation}
	\begin{array}{lcl}
	u_\rho(y) & = & w(y)^{\frac{2(n - 1)}{p'}} f^{p - 1}(\rho T^{-1}(y))\\
	v_\rho(x) & = & w(x)^{\frac{2n}{t'}}\left(E_Bf\right)\left(\rho T^{-1}(x)\right)
	\end{array}
\end{equation}
satisfy $u_\rho\in L^{\theta + 1}(\bdy \bb R_+^n)$, $v_\rho \in L^{\kappa + 1}(\bb R_+^n)$ and 
\begin{equation}
\label{eq:UHSSubcriticalELSystem}
	\begin{cases}
	u(y) 
	& 
	\displaystyle
	= 
	
	\int_{\bb R_+^n}\frac{x_n^\beta v^\kappa(x)}{\abs{x - y}^{n - \alpha}}
	w(x)^\tau w(y)^\sigma \; \d x
	\\
	v(x) 
	& 
	\displaystyle
	=  
	\int_{\bdy \bb R_+^n} \frac{x_n^\beta u^\theta(y)}{\abs{x - y}^{n - \alpha}}
	w(x)^\tau w(y)^\sigma \; \d y.  
	\end{cases}
\end{equation}
\end{lemma}
\begin{proof}
The integrability assertions on $u_\rho$ and $v_\rho$ follow immediately from the fact that $f\mapsto w^{\frac{2(n-1)}{p}} f\circ T^{-1}$ is an isometry $L^p(\bdy B) \to L^p(\bb R_+^n)$, the fact that $g\mapsto w^{\frac{2n}{t}} g\circ T^{-1}$ is an isometry $L^t(B)\to L^t(\bb R_+^n)$ and from estimate \eqref{eq:SubcriticalBallExtension}. To that $u_\rho$ and $v_\rho$ satisfy \eqref{eq:UHSSubcriticalELSystem}note that since $\rho$ is an isometry \eqref{eq:DistanceToBoundaryUnderPreimage} and \eqref{eq:DistanceBetweenPreimages} guarantee that for any $x\in \bb R_+^n$ and $y\in \bdy \bb R_+^n$
\begin{eqnarray}
\label{eq:KernelRotationTransformation}
	H(\rho T^{-1}(x), \rho T^{-1}(y))
	& = &  
	H(T^{-1}(x), T^{-1}(y))
	\notag\\
	& = & 
	\displaystyle
	\frac{x_n^\beta}{\abs{x - y}^{n - \alpha}}w(x)^{\alpha + 2\beta - n}w(y)^{\alpha - n}. 
\end{eqnarray}
Using the change of variable $\zeta = \rho T^{-1}(y)$ gives
\begin{eqnarray*}
	(E_Bf)(\rho T^{-1}(x))
	&  = & 
	\int_{\bdy B} H(\rho T^{-1}(x), \zeta)f(\zeta)\; \d S_\zeta
	\\
	& = & 
	\int_{\bdy \bb R_+^n}H(\rho T^{-1}(y), \rho T^{-1}(y)) f(\rho T^{-1}(y))w(y)^{2(n - 1)}\; \d y. 
\end{eqnarray*}
Combining this equation with \eqref{eq:KernelRotationTransformation} yields the second of the two asserted equalities. To prove the first of the two asserted equalities, use the change of variable $\xi = \rho T^{-1}(x)$ in equation \eqref{eq:BallExtremalEulerLagrange} to get
\begin{eqnarray*}
	f^{p - 1}(\rho T^{-1}(y))
	& = & 
	\int_B H(\xi, \rho T^{-1}(y)) \left(E_B f\right)(\xi)^{t' - 1}\; \d \xi
	\\
	& = & 
	\int_{\bb R_+^n} H(\rho T^{-1}(x), \rho T^{-1}(y))\left(E_B f\right)(\rho T^{-1}(x))^{t' - 1}w(x)^{2n}\; \d x. 
\end{eqnarray*}
Combing this equality with \eqref{eq:KernelRotationTransformation} yields the first of the two asserted equalities. 
\end{proof}
Theorem \ref{theorem:ConstantExtremalsBestEmbeddingConstant} is implied by the following proposition. 
\begin{prop}
\label{prop:RadialSymmetryOfExtremalSystemSolutions}
Let $n\geq 2$ and suppose $\alpha$ and $\beta$ satisfy \eqref{eq:MinimalAlphaBeta}. Suppose $\theta>0$ and $\kappa>0$ satisfy $(\theta + 1)^{-1}+ (\kappa + 1)^{-1} < 1$ and 
\begin{equation}
\label{eq:ThetaKappaBothSoftlySubcritical}
	\theta \leq \frac{n + \alpha - 2}{n - \alpha}, 
	\qquad
	\kappa \leq \frac{n + \alpha + 2\beta}{n - \alpha - 2\beta}
\end{equation}
and that at least one of the inequalities in \eqref{eq:ThetaKappaBothSoftlySubcritical} is strict. If $u\in L^{\theta + 1}(\bdy \bb R_+^n)$ and $v\in L^{\kappa + 1}(\bb R_+^n)$ are positive solutions to \eqref{eq:UHSSubcriticalELSystem} then $u$ and $v$ are symmetric about the $x_n$-axis. \\
\end{prop}
Before we give the proof of Proposition \ref{prop:RadialSymmetryOfExtremalSystemSolutions} let us show that Theorem \ref{theorem:ConstantExtremalsBestEmbeddingConstant} follows from Proposition \ref{prop:RadialSymmetryOfExtremalSystemSolutions}. 
\begin{proof}[Proof of Theorem \ref{theorem:ConstantExtremalsBestEmbeddingConstant}]
Suppose for sake of obtaining a contradiction that $f$ is not constant and let $\bar f = \abs{\bdy B}^{-1}\int_{\bdy B} f\; \d S_{\zeta}$. There is $\delta>0$ for which both of $A^\delta = \{\zeta\in \bdy B: f> \bar f + \delta\}$ and $A_\delta = \{\zeta\in \bdy B: f< \bar f - \delta\}$ have positive measure. Let $\rho:\bb R^n \to \bb R^n$ be a  rotation for which there exists a rotation $\phi$ about the $x_n$-axis satisfying $\abs{\phi\rho(A^\delta)\cap \rho(A_\delta)}_{\bdy B}> 0$. This gives $\abs{T(\phi\rho(A^\delta))\cap T(\rho(A_\delta))}_{\bdy \bb R_+^n}> 0$, where $T$ as in \eqref{eq:BallToUHSConformalMap}. Using the notation of \eqref{eq:RotatedBallToUHSTransformation}, if $y\in T(\phi\rho(A^\delta))$ then
\begin{equation*}
	u_{(\phi\rho)^{-1}}(y)
	= 
	w(y)^{\frac{2(n - 1)}{p'}}f^{p - 1}\left( (\phi\rho)^{-1}T^{-1}(y)\right)
	> 
	(\bar f + \delta)^{p - 1}w(y)^{\frac{2(n - 1)}{p'}}. 
\end{equation*}
On the other hand, since $T$ commutes with $\phi$ and since $u$ is symmetric about the $x_n$-axis, if $y\in T(\rho(A_\delta))$ then
\begin{eqnarray*}
	u_{(\phi\rho)^{-1}}(y)
	& = & 
	w(y)^{\frac{2(n - 1)}{p'}}f^{p -1}\left((\phi \rho)^{-1}T^{-1}(\phi y)\right)
	\\
	& = & 
	w(y)^{\frac{2(n - 1)}{p'}} f^{p - 1}\left(\rho^{-1}T^{-1}(y)\right)
	\\
	& < & 
	(\bar f - \delta)^{p - 1}w(y)^{\frac{2(n - 1)}{p'}}. 
\end{eqnarray*}
The previous two estimates imply that $\abs{T(\phi\rho(A^\delta))\cap T(\rho(A_\delta))}_{\bdy \bb R_+^n}= 0$, a contradiction. To compute $C_*(n, \alpha, \beta, p, t)$ simply observe that function $f\equiv (n\omega_n)^{-\frac 1p}$ has unit $L^p(\bdy B)$ norm so the proposed expression for $C_*$ holds. 
\end{proof}

The remainder of this section is devoted to the proof of Proposition \ref{prop:RadialSymmetryOfExtremalSystemSolutions}. We start by giving the following corollary to Proposition \ref{prop:SubcriticalBallHLSInequality} which will be used repeatedly. Its proof follows from the properties of $T$ together with the fact that $f\mapsto w^{2(n-1)/p} f\circ T^{-1}$ is an isometry of $L^p(\bdy B)$ into $L^p(\bdy \bb R_+^n)$ and the fact that $g\mapsto w^{2n/t} g\circ T^{-1}$ is an isometry of $L^t(B)$ into $L^t(\bdy \bb R_+^n)$. 
\begin{coro}
\label{coro:UHSHLSExtensionRestriction}
Let $n\geq 2$ and suppose $\beta\geq 0$ and $0< \alpha + \beta < n$.  Let $w$, $\sigma$ and $\tau$ be as in \eqref{eq:SubcriticalWeight} and \eqref{eq:SubcriticalWeightExponents}. 
\begin{enumerate}[(a)]
	\item If $p>1$, $t>1$ satisfy both $1< \frac 1 p + \frac 1 t$ and \eqref{eq:ptSubcritical} then 
\begin{equation*}
	\abs{
	\int_{\bb R_+^n}\int_{\bdy \bb R_+^n}
	\frac{x_n^\beta f(y) g(x)}{\abs{x - y}^{n - \alpha}}w(x)^\tau w(y)^\sigma
	\; \d y\; \d x
	}
	\leq 
	C_*(n, \alpha, \beta, p, t)\norm{f}_{L^p(\bdy \bb R_+^n)}\norm{g}_{L^t(\bb R_+^n)}
\end{equation*}
for all $f\in L^p(\bdy \bb R_+^n)$ and all $g\in L^t(\bb R_+^n)$, where  and $C_*(n, \alpha, \beta, p, t)$ is given in \eqref{eq:BallSubcriticalHLSSharpConstant}. 
	\item If $s>0$ and if $\frac{n -1}{n}\left(\frac 1 p - \frac{\alpha + \beta - 1}{n - 1}\right)< \frac 1 {s}< \frac 1 p < 1$ then for all $f\in L^p(\bdy \bb R_+^n)$, 
	\begin{equation*}
		\norm{
		\int_{\bdy \bb R_+^n}
		\frac{x_n^\beta f(y)}{\abs{x- y}^{n - \alpha}}
		w(x)^\tau w(y)^\sigma \; \d y
		}_{L^{s}(\bb R_+^n;\d x)}
		\leq
		C_*(n, \alpha, \beta, p, s') \norm{f}_{L^p(\bdy \bb R_+^n)}. 
	\end{equation*}
	\item If $r>0$ and if $\frac{n}{n - 1}\left(\frac 1 t- \frac{\alpha + \beta}{n}\right)< \frac 1 {r} < \frac 1 t < 1$ then for all $g\in L^t(\bb R_+^n)$, 
	\begin{equation*}
		\norm{
		\int_{\bb R_+^n}
		\frac{x_n^\beta g(x)}{\abs{x- y}^{n - \alpha}}w(x)^\tau w(y)^\sigma \; \d x
		}_{L^{r}(\bdy \bb R_+^n; \d y)}
		\leq
		C_*(n, \alpha, \beta, r', t)\norm{g}_{L^t(\bb R_+^n)}. 
	\end{equation*}
\end{enumerate}
\end{coro}

For $\lambda\in \bb R$ define
\begin{equation*}
	\Sigma_{\lambda, n-1} = \{y\in \bdy \bb R_+^n: y_1< \lambda\}, 
	\qquad
	\Sigma_{\lambda, n} = \{x\in \bb R_+^n: x_1< \lambda\}
\end{equation*}
and the reflected functions
\begin{equation*}
	u_\lambda(y) = u(y^\lambda)
	\qquad
	v_{\lambda}(x) = v(x^\lambda), 
\end{equation*}
where $y^\lambda = (2\lambda - y_1, y_2, \cdots, y_{n -1})$ and $x^\lambda = (2\lambda - x_1, x_2, \cdots, x_n)$. Define also 
\begin{equation*}
	\Sigma_{\lambda, n-1}^u 
	= 
	\{y\in \Sigma_{\lambda, n-1}: u(y)> u_\lambda(y)\}
	\qquad
	\text{ and }
	\qquad
	\Sigma_{\lambda, n}^v = 
	\{ x\in \Sigma_{\lambda, n}: v(x)>v_\lambda(x)\}. 
\end{equation*}
If $\lambda\leq 0$ and $(x,y)\in \Sigma_{\lambda, n}\times \Sigma_{\lambda,n-1}$ then 
\begin{equation}
\label{eq:LambdaNegativeConsequences}
	|y^\lambda + 2e_n|\leq \abs{ y + 2e_n},
	\quad 
	|x^\lambda - y|\geq \abs{x - y}
	\quad
	\text{ and }
	\quad
	|x^\lambda + 2e_n|\leq \abs{x + 2e_n}
\end{equation}
and these inequalities are strict when $\lambda< 0$. We also define the sets

By performing routine computations, one finds that if $u, v$ satisfy \eqref{eq:UHSSubcriticalELSystem} then 
\begin{equation}
\label{eq:uLambdaDifference}
\begin{array}{lcl}
	\multicolumn{3}{l}{
	\displaystyle
	u(y) - u_\lambda(y)
	}
	\\
	& = & 
	\displaystyle
	\int_{\Sigma_{\lambda, n}}
	x_n^\beta
	\left[
	\frac{v(x)^\kappa}{\abs{x - y}^{n - \alpha}}w(x)^\tau
	+
	\frac{v_\lambda(x)^\kappa}{\abs{x^\lambda -y}^{n - \alpha}}w(x^\lambda)^\tau
	\right]
	w(y)^\sigma\; \d x
	\\
	& & 
	\displaystyle
	-
	\int_{\Sigma_{\lambda, n}}
	x_n^\beta
	\left[
	\frac{v_\lambda(x)^\kappa}{\abs{x - y}^{n - \alpha}}w(x^\lambda)^\tau
	+
	\frac{v(x)^\kappa}{\abs{x^\lambda -y}^{n - \alpha}}w(x)^\tau
	\right]
	w(y^\lambda)^\sigma\; \d x
\end{array}
\end{equation}
and
\begin{equation}
\label{eq:vLambdaDifference}
\begin{array}{lcl}
	\multicolumn{3}{l}{
	v(x) - v_\lambda(x)
	}
	\\
	& = & 
	\displaystyle
	\int_{\Sigma_{\lambda, n-1}}
	x_n^\beta
	\left[
	\frac{u^\theta(y)}{\abs{x - y}^{n - \alpha}}w(y)^\sigma
	+
	\frac{u_\lambda^\theta(y)}{\abs{x^\lambda - y}^{n - \alpha}}w(y^\lambda)^\sigma
	\right]
	w(x)^\tau
	\; \d y
	\\
	& & 
	\displaystyle
	-
	\int_{\Sigma_{\lambda, n-1}}
	x_n^\beta
	\left[
	\frac{u_\lambda^\theta(y)}{\abs{x - y}^{n - \alpha}}w(y^\lambda)^\sigma
	+
	\frac{u^\theta(y)}{\abs{x^\lambda - y}^{n - \alpha}}w(y)^\sigma
	\right]
	w(x^\lambda)^\tau
	\; \d y. 
\end{array}
\end{equation}

The following lemma is the key to the moving planes process. 

\begin{lemma}
\label{lemma:KeyEstimatesForMovingPlanes}
Under the hypotheses of Proposition \ref{prop:RadialSymmetryOfExtremalSystemSolutions} there is a $\lambda$-independent constant $C_1>0$ such that for all $\lambda<0$ the following estimates hold. 
\begin{enumerate}[(a)]
	\item If $\theta\geq1$ and $\kappa \geq 1$ then 
	\begin{equation}
	\label{eq:u-uLambdaEstimateLargeExponents}
		\norm{u - u_\lambda}_{L^{\theta+1}(\Sigma_{\lambda, n-1}^u)}
		\leq
		C_1\norm{u}_{L^{\theta + 1}(\Sigma_{\lambda, n-1}^u)}^{\theta - 1}
		\norm{v}_{L^{\kappa + 1}(\Sigma_{\lambda, n}^v)}^{\kappa - 1}
		 \norm{u - u_\lambda}_{L^{\theta + 1}(\Sigma_{\lambda, n-1}^u)}
	\end{equation}
	and
	\begin{equation}
	\label{eq:v-vLambdaEstimateLargeExponents}
		\norm{v - v_\lambda}_{L^{\kappa + 1}(\Sigma_{\lambda, n}^v)}
		\leq
		C_1\norm{u}_{L^{\theta + 1}(\Sigma_{\lambda, n-1}^u)}^{\theta - 1}
		\norm{v}_{L^{\kappa + 1}(\Sigma_{\lambda, n}^v)}^{\kappa -1} 
		\norm{v - v_\lambda}_{L^{\kappa + 1}(\Sigma_{\lambda, n}^v)}. 
	\end{equation}
	\item If $\theta<1$ then for $1< \frac 1 \theta < r< \kappa$, 
	\begin{equation}
	\label{eq:v-vLambdaEstimateSmallTheta}
	\norm{v - v_\lambda}_{L^{\kappa + 1}(\Sigma_{\lambda, n}^v)}
	\leq
	C_1\norm{u}_{L^{\theta + 1}(\Sigma_{\lambda, n - 1}^u)}^{\theta - \frac 1 r}\norm{v}_{L^{\kappa +1}(\Sigma_{\lambda, n}^v)}^{\frac \kappa r - 1}\norm{v - v_\lambda}_{L^{\kappa + 1}(\Sigma_{\lambda, n}^v)}.
	\end{equation}
	\item If $\kappa < 1$ then for $1< \frac 1 \kappa < q < \theta$, 
	\begin{equation}
	\label{eq:u-uLambdaEstimateSmallKappa}
	\norm{u - u_\lambda}_{L^{\theta+1}(\Sigma_{\lambda, n-1}^u)}
		\leq
		C_1\norm{u}_{L^{\theta + 1}(\Sigma_{\lambda, n-1}^u)}^{\frac{\theta}{q} - 1}
		\norm{v}_{L^{\kappa + 1}(\Sigma_{\lambda, n}^v)}^{\kappa - \frac1q}
		 \norm{u - u_\lambda}_{L^{\theta + 1}(\Sigma_{\lambda, n-1}^u)}. 
	\end{equation}
\end{enumerate}
\end{lemma}
\begin{proof}
In terms of $\kappa$ and $\theta$ the subcritical weighting exponents are $\sigma = \frac{2(n - 1)}{\theta + 1} - n + \alpha \geq 0$ and $\tau = \frac{2n}{\kappa + 1} - n + \alpha + 2\beta\geq 0$. Let $\lambda\leq 0$. For a.e. $y\in \Sigma_{\lambda, n-1}$ \eqref{eq:uLambdaDifference} and \eqref{eq:LambdaNegativeConsequences} give
%
\begin{equation}
\label{eq:u-uLambdaLeftHalfEstimate1}
\begin{array}{lcl}
	\multicolumn{3}{l}{
	u(y) - u_\lambda(y)
	}
	\\
	&= & 
	\displaystyle
	\int_{\Sigma_{\lambda, n}}\frac{x_n^\beta}{\abs{x - y}^{n - \alpha}}
	\left(v^\kappa(x) w^\tau(x) w^\sigma(y) - v_\lambda^\kappa(x) w^\tau(x^\lambda) w^\sigma(y^\lambda) \right)\; \d x
	\\
	& & 
	\displaystyle
	+ \int_{\Sigma_{\lambda, n}}\frac{x_n^\beta}{\abs{x^\lambda - y}^{n - \alpha}}
	\left(v_\lambda^\kappa(x) w^\tau(x^\lambda) w^\sigma(y) - v^\kappa(x) w^\tau(x) w^\sigma(y^\lambda) \right)\; \d x
	\\
	& \leq & 
	\displaystyle
	\int_{\Sigma_{\lambda, n}}
	\left(\frac{x_n^\beta}{\abs{x - y}^{n - \alpha}} - \frac{x_n^\beta}{\abs{x^\lambda - y}^{n - \alpha}}\right)
	\left(v^\kappa(x) w^\tau(x) - v_\lambda^\kappa(x) w^\tau(x^\lambda)\right) w^\sigma(y)\; \d x
	\\
	&\leq & 
	\displaystyle
	\int_{\Sigma_{\lambda, n}}
	\left(\frac{x_n^\beta}{\abs{x - y}^{n - \alpha}} - \frac{x_n^\beta}{\abs{x^\lambda - y}^{n - \alpha}}\right)
	\left(v^\kappa(x) - v_\lambda^\kappa(x)\right)w^\tau(x) w^\sigma(y)\; \d x. 
\end{array}
\end{equation}
Continuing this estimate we have
\begin{eqnarray}
\label{eq:u-uLambdaLeftHalfEstimate}
	u(y) - u_\lambda(y)
	&\leq & 
	\int_{\Sigma_{\lambda, n}^v}
	\left(\frac{x_n^\beta}{\abs{x - y}^{n - \alpha}} - \frac{x_n^\beta}{\abs{x^\lambda - y}^{n - \alpha}}\right)
	\left(v^\kappa(x) - v_\lambda^\kappa(x)\right)w^\tau(x) w^\sigma(y)\; \d x
	\notag\\
	&\leq & 
	\int_{\Sigma_{\lambda, n}^v}
	\frac{x_n^\beta\left(v^\kappa(x) - v_\lambda^\kappa(x)\right)}{\abs{x - y}^{n - \alpha}}w^\tau(x) w^\sigma(y)\; \d x. 
\end{eqnarray}
By a similar computation \eqref{eq:vLambdaDifference} and \eqref{eq:LambdaNegativeConsequences} give 
\begin{equation}
\label{eq:v-vLambdaLeftHalfEstimate}
\begin{array}{lcl}
\multicolumn{3}{l}{
	v(x) - v_\lambda(x)
	}
	\\
	&\leq & 
	\displaystyle
	\int_{\Sigma_{\lambda, n -1}^u}
	\left( \frac{x_n^\beta}{\abs{x - y}^{n - \alpha}} - \frac{x_n^\beta}{\abs{x^\lambda - y}^{n - \alpha}}\right)
	(u^\theta(y) - u_\lambda^\theta(y))w^\tau(x) w^\sigma(x)\; \d y
	\\
	&\leq & 
	\displaystyle
	\int_{\Sigma_{\lambda, n -1}^u}\frac{x_n^\beta(u^\theta(y) - u_\lambda^\theta(y))}{\abs{x - y}^{n - \alpha}}w^\tau(x) w^\sigma(x)\; \d y
\end{array}
\end{equation}
for a.e. $x\in \Sigma_{\lambda, n}$. The assumption $(\kappa + 1)^{-1} + (\theta +1)^{-1}< 1$ guarantees that one of $\kappa$ or $\theta$ is strictly larger than $1$. We split the remainder of the proof into cases accordingly. \\
{\bf Case 1:} Assume $\kappa >1$. In this case, for $y\in \Sigma_{\lambda, n-1}^u$ the Mean Value Theorem and estimate \eqref{eq:u-uLambdaLeftHalfEstimate} give 
\begin{equation*}
	0
	< 
	u(y) - u_\lambda(y)
	\leq
	\kappa\int_{\Sigma_{\lambda, n}^v}
	\frac{x_n^\beta v^{\kappa - 1}(x)(v(x) - v_\lambda(x))}{\abs{x - y}^{n - \alpha}}w^\tau(x) w^\sigma(y)\; \d y. 
\end{equation*}
Applying Corollary \ref{coro:UHSHLSExtensionRestriction} (c) and H\"older's inequality gives
\begin{eqnarray}
\label{eq:u-uLambdaInitialEstimateKappaLarge}
	\norm{u - u_\lambda}_{L^{\theta+1}(\Sigma_{\lambda, n-1}^u)}
	&\leq & 
	C\norm{v^{\kappa -1}(v - v_\lambda)}_{L^{(\kappa + 1)/\kappa}(\Sigma_{\lambda, n}^v)}
	\notag
	\\
	& \leq & 
	C\norm{v}_{L^{\kappa + 1}(\Sigma_{\lambda, n}^v)}^{\kappa - 1}\norm{v - v_\lambda}_{L^{\kappa + 1}(\Sigma_{\lambda, n}^v)}
\end{eqnarray}
for some constant $C = C(n, \alpha, \beta, \kappa, \theta)>0$. \\
{\bf Case 1 (a):} Assume $\kappa>1$ and $\theta\geq 1$. In this case, for $x\in \Sigma_{\lambda, n}^v$ the Mean Value Theorem and \eqref{eq:v-vLambdaLeftHalfEstimate} give
\begin{equation*}
	0< v(x) - v_\lambda(x)
	\leq
	\theta\int_{\Sigma_{\lambda, n-1}^u}\frac{x_n^\beta u^{\theta - 1}(y)(u(y) - u_\lambda(y))}{\abs{x - y}^{n - \alpha}}w^\tau(x) w^\sigma(y)\; \d y. 
\end{equation*}
Applying Corollary \ref{coro:UHSHLSExtensionRestriction} (b) and H\"older's inequality gives
\begin{eqnarray}
\label{eq:v-vLambdaInitialEstimateThetaLarge}
	\norm{v - v_\lambda}_{L^{\kappa +1}(\Sigma_{\lambda, n}^v)}
	&\leq & 
	C\norm{u^{\theta - 1}(u - u_\lambda)}_{L^{(\theta + 1)/\theta}(\Sigma_{\lambda, n-1}^u)}
	\notag
	\\
	&\leq & 
	C\norm{u}_{L^{\theta + 1}(\Sigma_{\lambda, n-1})}^{\theta - 1}\norm{u - u_\lambda}_{L^{\theta + 1}(\Sigma_{\lambda, n-1}^u)}
\end{eqnarray}
for some constant $C = C(n, \alpha, \beta, \kappa, \theta)>0$. Combining this with estimate \eqref{eq:u-uLambdaInitialEstimateKappaLarge} gives estimate \eqref{eq:u-uLambdaEstimateLargeExponents}. Similarly, using \eqref{eq:u-uLambdaInitialEstimateKappaLarge} in \eqref{eq:v-vLambdaInitialEstimateThetaLarge} gives estimate \eqref{eq:v-vLambdaEstimateLargeExponents}. \\
{\bf Case 1(b):} Assume $\kappa>1$ and $\theta< 1$. Let $r$ satisfy $1< \frac 1 \theta < r< \kappa$. For $x\in \Sigma_{\lambda, n}^v$ the Mean Value Theorem and \eqref{eq:v-vLambdaLeftHalfEstimate} give
\begin{equation*}
	0
	< 
	v(x) - v_\lambda(x)
	\leq
	\theta r\int_{\Sigma_{\lambda, n-1}^u}
	\frac{x_n^\beta u(y)^{\theta - 1/r}\left(u^{1/r}(y) - u_\lambda^{1/r}(y)\right)}{\abs{x - y}^{n - \alpha}}w^\tau(x) w^\sigma(y)\; \d y. 
\end{equation*}
Applying Corollary \ref{coro:UHSHLSExtensionRestriction} (b) and H\"older's inequality gives
\begin{eqnarray}
\label{eq:v-vLambdaLebesgueEstimateThetaSmall}
	\norm{v - v_\lambda}_{L^{\kappa +1}(\Sigma_{\lambda, n}^v)}
	& \leq & 
	C\norm{u^{\theta - 1/r}\left(u^{1/r} - u_\lambda^{1/r}\right)}_{L^{(\theta + 1)/\theta}(\Sigma_{\lambda, n-1}^u)}
	\notag
	\\
	& \leq & 
	C\norm{u}_{L^{\theta +1}(\Sigma_{\lambda, n-1}^u)}^{\theta - 1/r}\norm{u^{1/r} - u_\lambda^{1/r}}_{L^{r(\theta+ 1)}(\Sigma_{\lambda, n - 1}^u)}. 
\end{eqnarray}
Now we estimate the Lebesgue norm of $u^{1/r} - u_\lambda^{1/r}$ appearing on the right hand side of this estimate. Define
\begin{eqnarray*}
	I_1(y) & = & \int_{\Sigma_{\lambda, n}^v}\frac{x_n^\beta v^\kappa(x)}{\abs{x - y}^{n - \alpha}}w(x)^\tau w(y)^\sigma\; \d x\\
	I_2(y) & = & \int_{\Sigma_{\lambda, n}^v}\frac{x_n^\beta v_\lambda^\kappa(x)}{\abs{x^\lambda - y}^{n - \alpha}}w(x^\lambda)^\tau w(y)^\sigma\; \d x\\
	I_3(y) & = & \int_{\Sigma_{\lambda, n}\setminus\Sigma_{\lambda, n}^v}\frac{x_n^\beta v^\kappa(x)}{\abs{x - y}^{n - \alpha}}w(x)^\tau w(y)^\sigma\; \d x\\
	I_4(y) & = & \int_{\Sigma_{\lambda, n}\setminus\Sigma_{\lambda, n}^v}\frac{x_n^\beta v_\lambda^\kappa(x)}{\abs{x^\lambda - y}^{n - \alpha}}w(x^\lambda)^\tau w(y)^\sigma\; \d x
\end{eqnarray*}
so that $u(y) = \sum_{j = 1}^4I_j(y)$. We claim that 
\begin{equation}
\label{eq:GoodSetClaim}
	I_3(y^\lambda) + I_4(y^\lambda) 
	\geq 
	I_3(y) + I_4(y)
	\qquad
	\text{ for a.e. } y\in \Sigma_{\lambda, n-1}. 
\end{equation}
To verify this claim, write
\begin{equation*}
	I_3(y) + I_4(y) - I_3(y^\lambda) - I_4(y^\lambda)
	= 
	\int_{\Sigma_{\lambda, n}\setminus \Sigma_{\lambda, n}^v} Q(x,y)\; \d x,
\end{equation*}
where the integrand $Q(x,y)$ satisfies
\begin{eqnarray*}
	x_n^{-\beta}Q(x,y)
	& = & 
	\frac{v^\kappa(x)}{\abs{x -y}^{n - \alpha}}w(x)^\tau w(y)^\sigma
	+ 
	\frac{v_\lambda^\kappa(x)}{\abs{x^\lambda - y}^{n - \alpha}}w(x^\lambda)^\tau w(y)^\sigma
	\\
	& & 
	- 
	\frac{v^\kappa(x)}{\abs{x - y^\lambda}^{n - \alpha}} w(x)^\tau w(y^\lambda)^\sigma
	- 
	\frac{v_\lambda^\kappa(x)}{\abs{x - y}^{n - \alpha}}w(x^\lambda)^\tau w(y^\lambda)^\sigma
	\\
	& = & 
	- \left(\frac{v^\kappa(x) w(x)^\tau}{\abs{x - y}^{n - \alpha}} + \frac{v_\lambda^\kappa(x) w(x^\lambda)^\tau}{\abs{x^\lambda - y}^{n - \alpha}}\right)
	\left(w(y^\lambda)^\sigma - w(y)^\sigma\right)\\
	& & 
	-\left(v_\lambda^\kappa(x) - v^\kappa(x)\right)w(x)^\tau w(y^\lambda)^\sigma\left(\abs{x - y}^{\alpha - n} - \abs{x- y^\lambda}^{\alpha - n}\right)
	\\
	& & 
	-v_\lambda^\kappa(x)\left(w(x^\lambda)^\tau - w(x)^\tau\right) w(y^\lambda)^\sigma\left(\abs{x - y}^{\alpha - n} - \abs{x - y^\lambda}^{\alpha - n}\right)
	\\
	& \leq & 
	0
\end{eqnarray*}
for a.e. $x\in \Sigma_{\lambda, n}\setminus \Sigma_{\lambda, n}^v$ and $y\in \Sigma_{\lambda, n-1}$. Estimate \eqref{eq:GoodSetClaim} follows. \\

Defining for $y\in \Sigma_{\lambda, n-1}^u$
\begin{equation*}
	a(y)
	= 
	\int_{\Sigma_{\lambda, n}^v}\frac{x_n^\beta v^\kappa (x)}{\abs{x - y}^{n - \alpha}}w(x)^\tau w(y)^\sigma\; \d x
	+ 
	\int_{\Sigma_{\lambda, n}^v}\frac{x_n^\beta v_\lambda^\kappa (x)}{\abs{x^\lambda - y}^{n - \alpha}}w(x)^\tau w(y)^\sigma\; \d x
\end{equation*}
and
\begin{equation*}
	b(y)
	=
	\int_{\Sigma_{\lambda, n}^v}\frac{x_n^\beta v_\lambda^\kappa(x)}{\abs{x - y}^{n - \alpha}}w(x)^\tau w(y)^\sigma\; \d x
	+
	\int_{\Sigma_{\lambda, n}^v}\frac{x_n^\beta v^\kappa(x)}{\abs{x^\lambda - y}^{n - \alpha}}w(x)^\tau w(y)^\sigma\; \d x, 
\end{equation*}
we have both 
\begin{equation*}
	u(y)
	\leq 
	a(y) + I_3(y) + I_4(y)
	+
	\int_{\Sigma_{\lambda, n}^v}\frac{x_n^\beta v_\lambda^\kappa(x)}{\abs{x- y}^{n - \alpha}}
	\left(w(x^\lambda)^\tau - w(x)^\tau\right) w(y^\lambda)^\sigma \; \d x
\end{equation*}
and
\begin{equation*}
	u_\lambda(y)
	\geq
	b(y) + I_3(y^\lambda) + I_4(y^\lambda)
	+ 
	\int_{\Sigma_{\lambda, n}^v}\frac{x_n^\beta v_\lambda^\kappa(x)}{\abs{x - y}^{n - \alpha}}\left(w(x^\lambda)^\tau - w(x)^\tau\right) w(y^\lambda)^\sigma\; \d x
\end{equation*}
for a.e. $y\in \Sigma_{\lambda, n-1}^u$. Since in addition, $u_\lambda(y)\geq b(y)$ in $\Sigma_{\lambda, n -1}^u$, there is a nonnegative function $c(y)$ for which 
\begin{equation*}
	u(y) - a(y) \leq c(y) \leq u_\lambda(y) - b(y)
\end{equation*}
for a.e. $y\in \Sigma_{\lambda, n - 1}^u$. Using $r>1$ and defining
\begin{eqnarray*}
	\varphi_1(x,y) 
	& = & 
	\left(\frac{x_n^\beta v^\kappa(x)}{\abs{x - y}^{n - \alpha}}w(x)^\tau w(y)^\sigma\right)^{1/r}\\
	\varphi_2(x,y) 
	& = & 
	\left(\frac{x_n^\beta v_\lambda^\kappa(x)}{\abs{x^\lambda - y}^{n - \alpha}}w(x)^\tau w(y)^\sigma\right)^{1/r}\\
	\psi_1(x,y) 
	& = & 
	\left(\frac{x_n^\beta v_\lambda^\kappa(x)}{\abs{x - y}^{n - \alpha}}w(x)^\tau w(y)^\sigma\right)^{1/r}\\
	\psi_2(x,y) 
	& = & 
	\left(\frac{x_n^\beta v^\kappa(x)}{\abs{x^\lambda - y}^{n - \alpha}}w(x)^\tau w(y)^\sigma\right)^{1/r}, 
\end{eqnarray*}	
we have
\begin{eqnarray}
\label{eq:uReflectionPowerDifference}
	u^{1/r}(y) - u_\lambda^{1/r}(y)
	& \leq & 
	\left(a(y) + c(y)\right)^{1/r} - \left(b(y) + c(y)\right)^{1/r}
	\notag
	\\
	&\leq & 
	a(y)^{1/r} - b(y)^{1/r}
	\notag
	\\
	& = & 
	\left(\int_{\Sigma_{\lambda, n}^v}\abs{(\varphi_1, \varphi_2)}_{\ell^r}^r\; \d x\right)^{1/r}
	- 
	\left(\int_{\Sigma_{\lambda, n}^v}\abs{(\psi_1, \psi_2)}_{\ell^r}^r\; \d x\right)^{1/r}
	\notag
	\\
	& \leq & 
	\left(\int_{\Sigma_{\lambda, n}^v}\abs{(\varphi_1 - \psi_1, \varphi_2- \psi_2)}_{\ell^r}^r\; \d x\right)^{1/r}
	\notag
	\\
	& = & 
	\left(\int_{\Sigma_{\lambda, n}^v}\left(\abs{\varphi_1 - \psi_1}^r + \abs{\varphi_2- \psi_2}^r\right)\; \d x\right)^{1/r}
\end{eqnarray}
for a.e. $y\in \Sigma_{\lambda, n-1}^u$. Moreover, for a.e. $(x,y)\in \Sigma_{\lambda, n}^v\times \Sigma_{\lambda, n-1}^u$ we have both 
\begin{equation*}
	\abs{\varphi_1 - \psi_1}^r
	= 
	(\varphi_1 - \psi_1)^r
	= 
	\frac{x_n^\beta}{\abs{x -y}^{n - \alpha}}\left(v^{\kappa/ r}(x)- v_\lambda^{\kappa/r}(x)\right)^r w(x)^\tau w(y)^\sigma
\end{equation*}
and
\begin{equation*}
	\abs{\varphi_2 - \psi_2}^r
	=
	\frac{x_n^\beta}{\abs{x^\lambda - y}^{n - \alpha}}\left(v^{\kappa/r}(x) - v_\lambda^{\kappa/r}(x)\right)^r w(x)^\tau w(y)^\sigma
	\leq
	(\varphi_1 - \psi_1)^r. 
\end{equation*}
Using these estimates in \eqref{eq:uReflectionPowerDifference} then applying the Mean Value Theorem we obtain  
\begin{eqnarray*}
	0
	& < & 
	u^{1/r}(y) - u_\lambda^{1/r}(y)
	\\
	&\leq & 
	2\left[
	\int_{\Sigma_{\lambda, n}^v}\frac{x_n^\beta}{\abs{x - y}^{n - \alpha}}\left(v^{\kappa/r}(x) - v_\lambda^{\kappa/r}(x)\right)^r w(x)^\tau w(y)^\sigma\; \d x
	\right]^{1/r}
	\\
	&\leq & 
	2\kappa\left[\int_{\Sigma_{\lambda, n}^v}\frac{x_n^\beta}{\abs{x - y}^{n - \alpha}}v^{\kappa - r}(x)\left(v(x) - v_\lambda(x)\right)^r w(x)^\tau w(y)^\sigma\; \d x
	\right]^{1/r}
\end{eqnarray*}
for a.e. $y\in \Sigma_{\lambda, n-1}^u$. An application of Corollary \ref{coro:UHSHLSExtensionRestriction} (c) followed by H\"older's inequality now gives
\begin{eqnarray*}
	\norm{u^{1/r} - u_\lambda^{1/r}}_{L^{r(\theta + 1)}(\Sigma_{\lambda, n - 1}^u)}
	& \leq& 
	C\norm{v^{\kappa - r}(v- v_\lambda)^r}_{L^{(\kappa + 1)/\kappa}(\Sigma_{\lambda, n}^v)}^{1/r}
	\\
	&\leq & 
	C\norm{v}_{L^{\kappa + 1}(\Sigma_{\lambda, n}^v)}^{\frac \kappa r - 1}\norm{v - v_\lambda}_{L^{\kappa + 1}(\Sigma_{\lambda, n}^v)}
\end{eqnarray*}
for some positive constant $C = C(n, \alpha, \beta, \theta, \kappa)$. Finally, using this estimate in \eqref{eq:v-vLambdaLebesgueEstimateThetaSmall} gives estimate \eqref{eq:v-vLambdaEstimateSmallTheta}. \\
{\bf Case 2:} Assume $\kappa \leq 1$. In this case performing computations similar in spirit to those carried out in Case 1(b) one finds that estimate \eqref{eq:u-uLambdaEstimateSmallKappa} holds. The details of this computation are omitted. 
\end{proof}
\begin{lemma}
\label{lemma:LambdaBarWellDefined}
Under the hypotheses of Proposition \ref{prop:RadialSymmetryOfExtremalSystemSolutions} there is $\lambda$ sufficiently negative such that for all $\mu\leq\lambda$, both 
\begin{equation}
\label{eq:ReflectionDominatesu}
	u_\mu \geq u \qquad \text{a.e.  in } \Sigma_{\mu, n-1}
\end{equation}
and
\begin{equation}
\label{eq:ReflectionDominatesv}
	v_\mu \geq v \qquad \text{a.e. in } \Sigma_{\mu, n}. 
\end{equation}
\end{lemma}
\begin{proof}
Since $u\in L^{\theta + 1}(\bdy \bb R_+^n)$ and $v\in L^{\kappa + 1}(\bb R_+^n)$, we have both $\norm{u}_{L^{\theta + 1}(\Sigma_{\mu, n-1})}\to 0$ and $\norm{v}_{L^{\kappa + 1}(\Sigma_{\mu, n})}\to 0$ as $\mu \to -\infty$. Therefore, if $\kappa\geq1$ and $\theta \geq1$, we may choose $\lambda$ sufficiently negative such that
\begin{equation*}
	C_1\norm{u}_{L^{\theta + 1}(\Sigma_{\mu, n-1})}^{\theta - 1}
	\norm{v}_{L^{\kappa + 1}(\Sigma_{\mu, n})}^{\kappa -1}
	\leq 
	\frac 12, 
\end{equation*}
whenever $\mu< \lambda$, where $C_1 = C_1(n, \alpha, \beta, \theta, \kappa)$ is the constant whose existence is guaranteed by Lemma \ref{lemma:KeyEstimatesForMovingPlanes}. In view of estimates \eqref{eq:u-uLambdaEstimateLargeExponents} and \eqref{eq:v-vLambdaEstimateLargeExponents}, if $\mu< \lambda $ we obtain both 
\begin{equation*}
	\norm{u - u_\mu}_{L^{\theta + 1}(\Sigma_{\mu, n-1}^u)}
	\leq
	\frac 12\norm{u - u_\mu}_{L^{\theta + 1}(\Sigma_{\mu, n-1}^u)}
\end{equation*}
and 
\begin{equation*}
	\norm{v - v_\mu}_{L^{\kappa + 1}(\Sigma_{\mu, n}^v)}
	\leq
	\frac 12 \norm{v - v_\mu}_{L^{\kappa + 1}(\Sigma_{\mu, n}^v)}. 
\end{equation*}
For any such $\mu$ we have $|\Sigma_{\mu, n-1}^u| = 0 = |\Sigma_{\mu, n}^v|$. If $\theta< 1$ then for any $\frac 1 \theta< r < \kappa$ there is $\lambda$ sufficiently negative such that 
\begin{equation*}
	C_1\norm{u}_{L^{\theta + 1}(\Sigma_{\mu, n-1})}^{\theta - 1/r}\norm{v}_{L^{\kappa + 1}(\Sigma_{\mu, n})}^{\frac{\kappa}{r} -1} 
	< 
	\frac 12
\end{equation*}
whenever $\mu< \lambda$. For any such $\lambda$ and $\mu$, estimate \eqref{eq:v-vLambdaEstimateSmallTheta} guarantees that $|\Sigma_{\mu, n}^v|= 0$. Estimate \eqref{eq:u-uLambdaInitialEstimateKappaLarge} now gives $|\Sigma_{\mu, n-1}^u| = 0$ for $\mu< \lambda$. Similarly, we find that if $\kappa< 1$ then both of $\Sigma_{\mu, n}^v$ and $\Sigma_{\mu, n-1}^u$ are measure zero sets. 
\end{proof}
Define 
\begin{equation*}
	\bar \lambda
	= 
	\sup\{\lambda< 0: \text{ both \eqref{eq:ReflectionDominatesu} and \eqref{eq:ReflectionDominatesv} hold for all } \mu\leq \lambda
	\}. 
\end{equation*}
\begin{lemma}
\label{lemma:SoftInequalityAtMaximalPlane}
If $\bar\lambda\leq 0$ then $u_{\bar\lambda}\geq u$ for a.e. in $\overline\Sigma_{\bar\lambda, n-1}$ and $v_{\bar\lambda}\geq v$ for a. e. in $\overline \Sigma_{\bar\lambda, n}$.
\end{lemma}
\begin{proof}
Suppose the assertion of the lemma is false so that either $|\Sigma_{\bar\lambda, n -1}^u|>0$ or $|\Sigma_{\bar\lambda, n}^v|>0$. In fact, if one of these inequalities holds then they both must hold. Indeed, if $|\Sigma_{\bar\lambda, n - 1}^u|>0$ then estimate \eqref{eq:u-uLambdaLeftHalfEstimate} implies that $|\Sigma_{\bar\lambda, n}^v|>0$. Similarly if $|\Sigma_{\bar\lambda, n}^v|>0$ then estimate \eqref{eq:v-vLambdaLeftHalfEstimate} guarantees that $|\Sigma_{\bar\lambda, n- 1}^u|>0$. Now choose $R>0$ (large) and $\delta>0$ (small) such that 
\begin{equation*}
	3\abs{\Sigma_{\bar\lambda, n-1}^u}
	\leq
	4\abs{ \left\{y\in B_R^{n - 1}\cap\Sigma_{\bar\lambda, n-1}: u(y) - u_{\bar\lambda}(y)>\delta\right\}}.
\end{equation*}
If $\lambda< \bar\lambda$ is sufficiently close to $\bar\lambda$ then 
\begin{equation}
\label{eq:SkinnyStrip}
	4\abs{B_R^{n - 1}\cap \{\lambda \leq y_1\leq \bar\lambda\}} \leq |\Sigma_{\bar\lambda, n -1}^u|.
\end{equation} 
For any $h\in C^0\cap L^{\theta + 1}(\bdy \bb R_+^n)$ there exists $\Lambda = \Lambda(h)< \bar\lambda$ such that for all $\Lambda < \lambda< \bar\lambda$ and a.e. $y\in B_R^{n - 1}\cap \Sigma_{\lambda, n-1} \cap\{u - u_{\bar\lambda}> \delta\}$, 
\begin{eqnarray*}
	\delta
	& < & 
	u(y) - u_\lambda(y) + u_\lambda(y) - u_{\bar\lambda}(y)
	\\
	&\leq & 
	|u(y^\lambda)  - h(y^\lambda)| + |h(y^\lambda) - h(y^{\bar\lambda})| + |h(y^{\bar\lambda}) - u(y^{\bar\lambda})|
	\\
	&\leq & 
	|u(y^\lambda)  - h(y^\lambda)| + |h(y^{\bar\lambda}) - u(y^{\bar\lambda})| + \frac \delta 2. 
\end{eqnarray*}
Choosing $h\in C^0\cap L^{\theta + 1}(\bdy \bb R_+^n)$ sufficiently close to $u$ in $L^{\theta + 1}$-norm then choosing $\lambda\in (\Lambda, \bar\lambda)$ sufficiently close to $\bar\lambda$ gives
\begin{equation*}
\begin{array}{lcl}
	\multicolumn{3}{l}{
	\displaystyle
	\abs{B_R^{n - 1}\cap \Sigma_{\lambda, n-1}\cap \{u- u_{\bar\lambda}> \delta\}}
	}
	\\
	&\leq & 
	\displaystyle
\abs{\left\{ y\in B_R^{n - 1}\cap\Sigma_{\lambda, n-1}: \frac\delta 2\leq |u(y^\lambda)  - h(y^\lambda)| + |h(y^{\bar\lambda}) - u(y^{\bar\lambda})|\right\}}	\\
	&\leq & 
	\displaystyle
	2\abs{\left\{ y\in \bdy \bb R_+^n: \frac\delta 4\leq \abs{u(y)  - h(y)}\right\}}
	\\
	&\leq & 
	\displaystyle
	C\delta^{-\theta - 1}\norm{u - h}_{L^{\theta + 1}(\bdy \bb R_+^n)}^{\theta  + 1}
	\\
	&\leq &
	\frac 14 \abs{\Sigma_{\bar\lambda, n - 1}^u}. 
\end{array}
\end{equation*}
Combining this estimate with \eqref{eq:SkinnyStrip} gives
\begin{eqnarray*}
	\frac 34 \abs{\Sigma_{\bar\lambda, n - 1}^u}
	& \leq & 
	\abs{B_R^{n - 1}\cap \{\lambda\leq y_1\leq \bar\lambda\}}
	+ 
	\abs{B_R^{n -1}\cap \Sigma_{\lambda, n - 1}\cap \{u - u_{\bar\lambda}> \delta\}}
	\\
	&\leq & 
	\frac12 \abs{\Sigma_{\bar\lambda, n - 1}^u}
\end{eqnarray*}
whenever $\lambda< \bar\lambda$ is sufficiently close to $\bar\lambda$, a contradiction. 
\end{proof}
\begin{lemma}
\label{lemma:BothStrictlyPositive}
Under the hypotheses of Proposition \ref{prop:RadialSymmetryOfExtremalSystemSolutions}, if $\bar\lambda< 0$ then $u_{\bar\lambda}> u$ a.e. in $\Sigma_{\bar\lambda, n - 1}$ and $v_{\bar\lambda}> v$ a.e. in $\Sigma_{\bar\lambda, n}$. 
\end{lemma}
\begin{proof}
Suppose for the sake of obtaining a contradiction that $\bar\lambda< 0$ and that there is a positive-measure subset $A\subset \Sigma_{\bar\lambda, n-1}$ on which the equality $u_{\bar\lambda}= u$. For all $y\in A$ estimate \eqref{eq:u-uLambdaLeftHalfEstimate1} and the fact that $v_{\bar\lambda}\geq v$ a.e. in  $\Sigma_{\bar\lambda, n}$ give
\begin{eqnarray*}
	0 
	& = & 
	u(y)- u_{\bar\lambda}(y) 
	\\
	& \leq & 
	\int_{\Sigma_{\bar\lambda, n}}
	\left(\frac{x_n^\beta}{\abs{x - y}^{n -\alpha}} - \frac{x_n^\beta}{|x^{\bar\lambda} - y|^{n - \alpha}}\right)
	\left(v^\kappa(x) - v_{\bar\lambda}^\kappa(x)\right)
	w^\tau(x) w^\sigma(y)\; \d x
	\\
	& \leq & 
	0. 
\end{eqnarray*}
Consequently, $v_{\bar\lambda} = v$ a.e. in $\Sigma_{\bar\lambda, n}$. Using this equality in the second-to-last line of estimate \eqref{eq:u-uLambdaLeftHalfEstimate1}, for all $y\in A$ we obtain 
\begin{eqnarray*}
	0 
	& = & 
	u(y)- u_{\bar\lambda}(y) 
	\\
	& \leq & 
	\int_{\Sigma_{\bar\lambda, n}}
	\left(\frac{x_n^\beta}{\abs{x - y}^{n -\alpha}} - \frac{x_n^\beta}{|x^{\bar\lambda} - y|^{n - \alpha}}\right)
	\left(v^\kappa(x)w^\tau(x) - v_{\bar\lambda}^\kappa(x)w^\tau(x^{\bar\lambda})\right)
	w^\sigma(y)\; \d x
	\\
	& = & 
	\int_{\Sigma_{\bar\lambda, n}}
	\left(\frac{x_n^\beta}{\abs{x - y}^{n -\alpha}} - \frac{x_n^\beta}{|x^{\bar\lambda} - y|^{n - \alpha}}\right)
	v^\kappa(x)\left(w^\tau(x) - w^\tau(x^{\bar\lambda})\right)
	w^\sigma(y)\; \d x
	\\
	& < & 
	0, 
\end{eqnarray*}
a contradiction. By a similar argument one can verify that $v_{\bar\lambda}>v$ a.e. in $\Sigma_{\bar\lambda, n}$. 
\end{proof}
\begin{lemma}
\label{lemma:PlaneLambdaBar=0}
Under the hypotheses of Proposition \ref{prop:RadialSymmetryOfExtremalSystemSolutions} $\bar \lambda = 0$. 
\end{lemma}
\begin{proof}
We will show both that
\begin{equation}
\label{eq:MovingPlaneLambdaLimitu}
	\lim_{\lambda\to\bar\lambda^+}\norm{u}_{L^{\theta + 1}(\Sigma_{\lambda, n-1}^u)} = 0
\end{equation}
and that
\begin{equation}
\label{eq:MovingPlaneLambdaLimitv}
	\lim_{\lambda\to\bar\lambda^+}\norm{v}_{L^{\kappa + 1}(\Sigma_{\lambda, n}^v)} = 0. 
\end{equation}
Under the assumption that these limits hold Lemma \ref{lemma:KeyEstimatesForMovingPlanes} guarantees the existence of $\delta>0$ sufficiently small such that if $\lambda\leq \bar\lambda + \delta$ then 
\begin{equation*}
	\norm{u- u_\lambda}_{L^{\theta + 1}(\Sigma_{\lambda, n -1}^u)}
	\leq
	\frac 12
	\norm{u - u_\lambda}_{L^{\theta + 1}(\Sigma_{\lambda, n - 1}^u)}
\end{equation*}
and 
\begin{equation*}
	\norm{v - v_\lambda}_{L^{\kappa + 1}(\Sigma_{\lambda, n}^v)}
	\leq
	\frac 12
	\norm{v - v_\lambda}_{L^{\kappa + 1}(\Sigma_{\lambda, n}^v)}.
\end{equation*}
These inequalities imply that both $\Sigma_{\lambda,n -1}^u$ and $\Sigma_{\lambda. v}^n$ are zero-measure sets whenever $\bar\lambda \leq \lambda \leq \bar\lambda + \delta$ thus contradicting the maximality of $\bar\lambda$.\\ 

The remainder of the proof of Lemma \ref{lemma:PlaneLambdaBar=0} is devoted to showing that \eqref{eq:MovingPlaneLambdaLimitu} and \eqref{eq:MovingPlaneLambdaLimitv} hold. Since the proofs of these limits are similar, only the details of \eqref{eq:MovingPlaneLambdaLimitu} will be presented. If $\epsilon>0$ is given we may choose $R>0$ large then choose $\eta = \eta(R)>0$ small such that for all $\lambda\in [\bar\lambda, \bar\lambda + \eta]$, 
\begin{eqnarray*}
	\norm{v}_{L^{\kappa + 1}(\Sigma_{\lambda, n}^v)}^{\kappa + 1}
	& \leq & 
	\norm{v}_{L^{\kappa + 1}(\bb R_+^n\setminus B_R)}^{\kappa + 1}
	+
	\norm{v}_{L^{\kappa + 1}(B_R^+\cap\{\bar\lambda - \eta\leq x_1\leq \bar\lambda + \eta\})}^{\kappa + 1}
	\\
	& & 
	+
	\norm{v}_{L^{\kappa + 1}(B_R^+\cap\{x_n< \eta\})}^{\kappa + 1}
	+
	\norm{v}_{L^{\kappa + 1}(\Sigma_{\lambda, n}^v(R, \eta))}^{\kappa + 1}
	\\
	&\leq & 
	\norm{v}_{L^{\kappa + 1}(\Sigma_{\lambda, n}^v(R, \eta))}^{\kappa + 1} + \epsilon, 
\end{eqnarray*}
where we use the notation 
\begin{equation*}
	\Sigma_{\lambda, n}^v(R, \eta) 
	= 
	\Sigma_{\lambda, n}^v \cap B_R\cap \{x_1< \bar\lambda - \eta, x_n> \eta\}. 
\end{equation*}
Thus, to establish \eqref{eq:MovingPlaneLambdaLimitv} it is sufficient to show that for all $R>0$ large and all $\eta>0$ small, 
\begin{equation}
\label{eq:vBadSetMeasureToZero}
	\abs{\Sigma_{\lambda, n}^v(R, \eta)}
	\to 
	0
\end{equation}
as $\lambda \to \bar\lambda^+$. Suppose for the sake of obtaining a contradiction that \eqref{eq:vBadSetMeasureToZero} fails and choose $R>0$,  $\eta>0$, $\epsilon_0>0$ and a sequence $\lambda_k\to \bar\lambda^+$ for which $\abs{\Sigma_{\lambda_k, n}^v(R, \eta)}> \epsilon_0$ for all $k$. The first inequality in estimate \eqref{eq:v-vLambdaLeftHalfEstimate} guarantees the existence of a positive constant $c_1>0$ depending only on $n, \alpha, \beta, R, \eta$ and the distribution function of $u_{\bar\lambda}^\theta - u^\theta$ such that $v_{\bar\lambda} - v\geq c_1$ for a.e. $x\in \Sigma_{\bar\lambda - \eta, n}\cap B_R\cap\{x_n>\eta\}$. 
Consequently, 
\begin{equation*}
	v_{\bar\lambda} - v_{\lambda_k}
	\geq 
	v_{\bar\lambda} - v
	\geq 
	c_1
\end{equation*}
a.e. in $\Sigma_{\lambda_k, n}^v(R, \eta)$. For any $h\in C^0\cap L^{\kappa + 1}(\bb R_+^n)$ there is $\Lambda = \Lambda(h)>\bar\lambda$ such that if $\lambda_k\in [\bar\lambda , \Lambda)$ then  
\begin{eqnarray*}
	c_1
	&\leq & 
	v_{\bar\lambda}(x) - v_{\lambda_k}(x)
	\\
	&\leq & 
	|v(x^{\bar\lambda}) - h(x^{\bar\lambda})|
	+
	|h(x^{\lambda_k}) - v(x^{\lambda_k})|
	+
	\frac{c_1}{2}
\end{eqnarray*}
for a.e. $x\in \Sigma_{\lambda_k, n}^v(R, \eta)$. Choosing $\norm{h - v}_{L^{\kappa + 1}(\bb R_+^n)}$ small (depending on $c_1$ and $\epsilon_0$) then choosing $k = k(h)$ large we get
\begin{eqnarray*}
	\epsilon_0
	&< & 
	\abs{\Sigma_{\lambda_k, n}^v(R, \eta)}
	\\
	&\leq & 
	\abs{
	\left\{
	x\in \bb R_+^n: \frac{c_1}{2}\leq 
	|v(x^{\bar\lambda}) - h(x^{\bar\lambda})|
	+
	|h(x^{\lambda_k}) - v(x^{\lambda_k})|
	\right\}}
	\\
	&\leq & 
	2\abs{
	\left\{
	x\in \bb R_+^n: \frac{c_1}{4}\leq
	|v(x)- h(x)|\right\}}
	\\
	&\leq & 
	Cc_1^{-\kappa - 1}\norm{h - v}_{L^{\kappa + 1}(\bb R_+^n)}
	\\
	&\leq & 
	\frac{\epsilon_0}{2}, 
\end{eqnarray*}
a contradiction. This establishes \eqref{eq:vBadSetMeasureToZero} and hence \eqref{eq:MovingPlaneLambdaLimitv}. 
\end{proof}
\begin{proof}[Proof of Proposition \ref{prop:RadialSymmetryOfExtremalSystemSolutions}]
It suffices to show that for every direction $e\subset \bb S^{n - 2}\subset \bdy\bb R_+^n$, the inequalities
\begin{equation}
\label{eq:ReflectionInequalitiesArbitraryDirection}
	u \leq u\circ R_{e} 
	\qquad\text{ and }\qquad
	v\leq v\circ R_{e}  
\end{equation}
hold in $L^{\theta + 1}(\{y\in \bdy \bb R_+^n :  y\cdot e\geq 0\})$ and in $L^{\kappa + 1}(\{x\in \bb R_+^n: x\cdot e \geq 0\})$ respectively, where $R_e: \bb R^n\to \bb R^n$ is reflection about the hyperplane $\{x\in \bb R^n: x\cdot e = 0\}$. The equality $\bar\lambda = 0$ and Lemma \ref{lemma:SoftInequalityAtMaximalPlane} guarantee that inequalities \eqref{eq:ReflectionInequalitiesArbitraryDirection} hold for $e = - e_1 = (-1, 0, \ldots, 0)$. Arguing similarly to the case $e = -e_1$ one can show that inequalities \eqref{eq:ReflectionInequalitiesArbitraryDirection} hold for all $e\in \bb S^{n - 2}$.\\
\end{proof}
%

\section{Sharp Inequality and Extremal Functions for the Conformally Invariant Exponents}
\label{section:ConformallyInvariantExponents}
%
\subsection{Sharp inequality on $\bb R_+^n$}
\begin{proof}[Proof of Theorem \ref{theorem:ConformallyInvariantSharpHalfSpaceInequality}]
First we show that the sharp inequality 
\begin{equation}
\label{eq:SharpConformallyInvariantBallInequality}
	\norm{E_Bf}_{L^{\frac{2n}{n - \alpha - 2\beta}}(B)}
	\leq
	C_e(n, \alpha, \beta)\norm{f}_{L^{\frac{2(n - 1)}{n +\alpha -2}}(\bdy B)}
\end{equation}
holds for every $f\in L^{\frac{2(n -1)}{n +\alpha - 2}}(\bdy B)$. It suffices to establish \eqref{eq:SharpConformallyInvariantBallInequality} for nonnegative $f\in C^0(\bdy B)$. 
Fix any such $f$. For any $\zeta\in \bdy B$ we have $f^p(\zeta) \to f^{\frac{2(n - 1)}{n +\alpha - 2}}(\zeta)$ as $p\to\left(\frac{2(n -1)}{n +\alpha - 2}\right)^+$. Additionally, for all $\frac{2(n - 1)}{n +\alpha -2}<p \leq \frac{2(n - 1)}{n +\alpha -2} + 1$ there holds
\begin{equation*}
	f^p(\zeta)
	\leq (1 + \max_{\bdy B} f)^{1 + \frac{2(n - 1)}{n + \alpha - 2}}, 
\end{equation*}
so the Bounded Convergence Theorem gives $\norm{f}_{L^p(\bdy B)}\to \norm{f}_{L^{\frac{2(n-1)}{n +\alpha - 2}}(\bdy B)}$ as $p\to\left(\frac{2(n -1)}{n +\alpha - 2}\right)^+$. We apply a similar argument to $\norm{E_Bf}_{L^{t'}(B)}$ as $t'\to \left(\frac{2n}{n - \alpha - 2\beta}\right)^-$. By Lemma \ref{lemma:MappingPropertiesBallExtension} of the appendix we have $E_Bf\in L^{2n/(n - \alpha - 2\beta)}(B)$ and therefore the following inequality holds for a.e. $\xi \in B$:
\begin{equation*}
	(E_Bf(\xi))^{t'} \leq (1 + E_Bf(\xi))^{2n/(n - \alpha - 2\beta)}\in L^1(B).
\end{equation*}
The Dominated Convergence Theorem gives $\norm{E_Bf}_{L^{t'}(B)}\to \norm{E_Bf}_{L^{2n/(n - \alpha - 2\beta)}(B)}$ as $t'\to (2n/(n - \alpha - 2\beta))^-$. Now, if $p> \frac{2(n -1)}{n +\alpha - 2}$ and $t'<\frac{2n}{n - \alpha - 2\beta}$ then Corollary \ref{coro:SubcriticalBallExtensionRestriction} (a) with optimal constant $C_*(n, \alpha, \beta, p, t)= (n\omega_n)^{-1/p}\norm{E_B 1}_{L^{t'}(B)}$ whose value was computed in Theorem \ref{theorem:ConstantExtremalsBestEmbeddingConstant} gives
\begin{equation*}
	\norm{E_Bf}_{L^{t'}(B)}
	\leq
	C_*(n, \alpha, \beta, p, t)\norm{f}_{L^p(\bdy B)}. 
\end{equation*}
Applying the same argument to the constant function $1$ that was just applied to $f$ we find that $C_*(n, \alpha, \beta, p, t)\to C_e(n, \alpha, \beta)$ as $p\to \left(\frac{2(n - 1)}{n +\alpha - 2}\right)^+$ and $t'\to \left(\frac{2n}{n - \alpha - 2\beta}\right)^-$. Thus, letting $p\to \left(\frac{2(n - 1)}{n +\alpha - 2}\right)^+$ and $t'\to \left(\frac{2n}{n - \alpha - 2\beta}\right)^-$ we recover \eqref{eq:SharpConformallyInvariantBallInequality}. To see that $C_e(n, \alpha, \beta)$ is the optimal constant observe first that \eqref{eq:SharpConformallyInvariantBallInequality} guarantees that 
\begin{equation*}
	C_e(n, \alpha, \beta)
	\geq
	\sup\left\{
	\norm{E_Bf}_{L^{\frac{2n}{n -\alpha - 2\beta}}(B)}: 
	\norm{f}_{L^{\frac{2(n - 1)}{n +\alpha - 2}}(\bdy B)} = 1
	\right\}. 
\end{equation*}
The reverse inequality is guaranteed by the fact that the constant function $f = (n\omega_n)^{-\frac{n +\alpha -2}{2(n - 1)}}$ has unit $L^{\frac{2(n-1)}{n +\alpha -2}}(\bdy B)$-norm and attains $C_e(n, \alpha, \beta)$. \\

Finally, given $f\in L^{\frac{2(n - 1)}{n + \alpha - 2}}(\bdy \bb R_+^n)$ we define $F:\bdy B\to \bb R$ by $F(\zeta) = \left(\frac2{\abs{\zeta +e_n}}\right)^{n + \alpha - 2}f\circ T(\zeta)$, where $T$ is as in \eqref{eq:BallToUHSConformalMap}. By performing elementary computations one can verify that $\norm{F}_{L^{\frac{2(n - 1)}{n + \alpha - 2}}(\bdy B)} = \norm{f}_{L^{\frac{2(n - 1)}{n + \alpha - 2}}(\bdy \bb R_+^n)}$ and that $\norm{E_BF}_{L^{\frac{2n}{n -\alpha - 2\beta}}(B)} = \norm{Ef}_{L^{\frac{2n}{n - \alpha - 2\beta}}(\bb R_+^n)}$. The assertion of the theorem follows. 
\end{proof}
\subsection{Classification of Extremal functions}
In this subsection we will prove Theorem \ref{theorem:UHSConformalExtremalClassification}, the classification of extremal functions in inequality \eqref{eq:SharpConformallyInvariantUHSInequality}. For the remainder of this section we will use the notation
\begin{equation*}
	p = \frac{2(n - 1)}{n + \alpha -2}
	\qquad
	t = \frac{2n}{n +\alpha + 2\beta}. 
\end{equation*}
With this notation we have $\theta = p' - 1 = \frac{n + \alpha - 2}{n - \alpha}$ and $\kappa = \frac{n + \alpha + 2\beta}{n -\alpha - 2\beta}$. To prove Theorem \ref{theorem:UHSConformalExtremalClassification} we classify all positive solutions $f\in L^p(\bdy \bb R_+^n)$ to equation \eqref{eq:ConformallyInvariantELHalfSpace}. If $f$ is any such function then the functions $u(y) = f(y)^{\frac{n - \alpha}{n + \alpha - 2}}$ and $v(x) = Ef(x)$ satisfy the system 
\begin{equation}
\label{eq:ConformallyInvariantExtremalSystem}
\begin{cases}
	\displaystyle
	u(y) 
	= 
	\int_{\bb R_+^n}\frac{x_n^\beta v^\kappa(x)}{\abs{x - y}^{n - \alpha}} \; \d x
	& 
	\text{ for } y\in \bdy \bb R_+^n
	\\
	\displaystyle
	v(x) 
	=
	\int_{\bdy \bb R_+^n}\frac{x_n^\beta u^\theta (y)}{\abs{x - y}^{n - \alpha}} \; \d y
	&
	\text{ for } x\in \bb R_+^n,    
\end{cases}
\end{equation}
where here and throughout the remainder of this section we use the simplified notation $Ef = E_{\alpha, \beta}f$. Theorem \ref{theorem:UHSConformalExtremalClassification} is implied by the following theorem. 
\begin{theorem}
\label{theorem:ExtremalSystemClassification}
Let $n\geq 2$ and suppose $\alpha$ and $\beta$ satisfy \eqref{eq:MinimalAlphaBeta} and \eqref{eq:AlphaBetaSubAffine}. If $u\in L^{\theta + 1}(\bdy \bb R_+^n)$ and $v\in L^{\kappa + 1}(\bb R_+^n)$ are positive solutions to \eqref{eq:ConformallyInvariantExtremalSystem} then there exists $c_1>0$, $d> 0$ and $y_0\in \bdy \bb R_+^n$ such that
\begin{equation}
\label{eq:uClassified}
	u(y) = \frac{c_1}{\left( d^2 + \abs{y - y_0}^2\right)^{\frac{n - \alpha}{2}}}
	\qquad
	\text{ for all } y\in \bdy \bb R_+^n. 
\end{equation}
\end{theorem}
Let us show that Theorem \ref{theorem:UHSConformalExtremalClassification} follows from Theorem \ref{theorem:ExtremalSystemClassification}. 
\begin{proof}[Proof of Theorem \ref{theorem:UHSConformalExtremalClassification}]
Equality \eqref{eq:ConformalBoundaryFunctionExtremal} follows immediately from equation \eqref{eq:uClassified} and the relation $f = u^{\frac{n + \alpha - 2}{n - \alpha}}$. It remains to show that the asserted limits of $Ef$ hold at $x_n = 0$. Note first that with $f$ as in \eqref{eq:ConformalBoundaryFunctionExtremal}, $Ef$ is continuous on $\bb R_+^n$. If $\alpha < 1$ then using the change of variable $y\mapsto (y - x')/x_n$ we obtain 
\begin{equation*}
	x_n^{1 -\alpha - \beta} Ef(x', x_n)
	= 
	\int_{\bdy \bb R_+^n} \frac{f(x'+ x_ny)}{(1 +\abs y^2)^{\frac{n - \alpha}{2}}}\; \d y
\end{equation*}
for all $x\in \bb R_+^n$. Since $f\in C^0\cap L^\infty(\bdy \bb R_+^n)$ the Dominated Convergence Theorem guarantees that 
\begin{equation*}
	\lim_{x_n\to 0^+}x_n^{1 -\alpha - \beta} Ef(x', x_n)
	=
	f(x') \int_{\bdy \bb R_+^n}\left(1 + \abs y^2 \right)^{\frac{\alpha - n}{2}}\; \d y. 
\end{equation*}	
Assertion (a) of Theorem \ref{theorem:UHSConformalExtremalClassification} is established. \\

To establish item (b), let $R> 3(\abs{x'} + \abs{y^0 - x'}) + 1$ and set 
\begin{eqnarray*}
	I_R(x') & = & \int_{B_R^{n - 1}(x')}\frac{f(y)}{\abs{x - y}^{n - 1}}\; \d y\\
	J_R(x') & = & \int_{\bdy \bb R_+^n\setminus B_R(x')}\frac{f(y)}{\abs{x - y}^{n - 1}}\; \d y
\end{eqnarray*}
so that
\begin{equation}
\label{eq:Alpha=1BoundaryEstimateSplit}
	x_n^{-\beta} Ef(x', x_n) = I_R(x') + J_R(x'). 
\end{equation}
For $y\in \bdy \bb R_+^n\setminus B_R(x')$ the inequalities $2\abs{x - y}\geq\abs y$ and $3\abs{y - y^0}\geq \abs y$ hold so
\begin{eqnarray*}
	J_R(x')
	&\leq & 
	C\int_{\bdy \bb R_+^n\setminus B_R(x')}\abs{y - y_0}^{1 - n}\abs{x - y}^{1 - n}\; \d y
	\\
	&\leq & 
	C\int_{\bdy \bb R_+^n\setminus B_1}\abs{y}^{-2(n - 1)}\; \d y
	\\
	&\leq & 
	C.
\end{eqnarray*}
To estimate $I_R(x')$ use the change of variable $y\mapsto(y- x')/x_n$ and the fact that $\abs{\Grad f}\in L^\infty(\bdy \bb R_+^n)$ to get
\begin{eqnarray*}
	I_R(x')
	&= & 
	\int_{B_{R/x_n}^{n - 1}}\frac{f(x' + x_ny)}{(1 + \abs y^2)^{\frac{n - 1}{2}}}\; \d y
	\\
	& = & 
	f(x') \int_{B_{R/x_n}^{n - 1}}\frac{\d y}{(1 + \abs y^2)^{\frac{n - 1}{2}}}
	+
	O(1)\int_{B_{R/x_n}^{n - 1}}\frac{x_n\abs y}{(1 + \abs y^2)^{\frac{n - 1}{2}}}\; \d y. 
\end{eqnarray*}
Since 
\begin{equation*}
	\int_{B_{R/x_n}^{n - 1}}\frac{x_n\abs y}{(1 + \abs y^2)^{\frac{n - 1}{2}}}\; \d y
	\leq 
	CR
\end{equation*}
and since 
\begin{equation*}
	-\frac 1{\log x_n}\int_{B_{R/x_n}}(1 + \abs{y}^2)^{\frac{1 - n}{2}}\; \d y
	\to 
	(n - 1)\omega_{n - 1}
\end{equation*}
as $x_n\to 0^+$, bringing the estimates of $I_R(x')$ and $J_R(x')$ back into \eqref{eq:Alpha=1BoundaryEstimateSplit} gives
\begin{equation*}
	-\lim_{x_n\to 0^+}\frac{1}{x_n^\beta \log x_n}Ef(x', x_n)
	= 
	(n - 1)\omega_{n - 1}f(x'). 
\end{equation*}
Item (b) is established. \\

To establish assertion (c) of Theorem \ref{theorem:UHSConformalExtremalClassification} observe that if $\alpha>1$ then we have
\begin{equation*}
	\lim_{x_n\to 0}x_n^{-\beta} Ef(x', x_n)
	= 
	\int_{\bdy \bb R_+^n} \frac{f(y)}{\abs{x' - y}^{n - \alpha}}\; \d y. 
\end{equation*}
Thus, assertion (c) will be established once we show that, up to a positive scalar multiple, 
\begin{equation*}
	f(x')^{\frac{n - \alpha}{n + \alpha - 2}}
	= 
	\int_{\bdy \bb R_+^n} \frac{f(y)}{\abs{x' - y}^{n - \alpha}}\; \d y. 
\end{equation*}
Set
\begin{equation*}
	\tilde v(x')
	=
	\int_{\bdy \bb R_+^n} \frac{f(y)}{\abs{x' - y}^{n - \alpha}}\; \d y. 
\end{equation*}
Note first that with $w$ as in \eqref{eq:SubcriticalWeight} and $T$ as in \eqref{eq:BallToUHSConformalMap}, the function 
\begin{equation*}
	\varphi(x')
	= 
	\int_{\bdy \bb R_+^n}\frac{w(y)^{n + \alpha - 2}}{\abs{x' - y}^{n - \alpha}}\; \d y
\end{equation*}
satisfies
\begin{equation*}
	\left(\frac 2{\abs{\xi + e_n}}\right)^{n - \alpha}
	\varphi \circ T(\xi)
	= 
	\int_{\bdy B}\frac{\d S_{\zeta}}{\abs{\xi - \zeta}^{n - \alpha}}
	=
	\int_{\bdy B}\frac{\d S_\zeta}{\abs{e_n - \zeta}^{n - \alpha}}
	= 
	C(n, \alpha)
\end{equation*}
for $\xi = T^{-1}(x')\in \bdy B$. In particular, $\varphi(x') = C(n, \alpha) w(x')^{n - \alpha}$. On the other hand, using the change of variable $y \mapsto 2(y - y_0)/d$ we find that $\tilde v(dx'/2 + y_0) = C(n, \alpha, d)\varphi(x')$. Therefore, up to a positive constant multiple we have
\begin{equation*}
	\tilde v(x') 
	= 
	\varphi(2(x' - y_0)/d)
	= 
	\left(d^2 + \abs{x' - y_0}^2\right)^{\frac{\alpha - n}{2}}, 
\end{equation*}
which is the desired equality. 
\end{proof}
The remainder of this subsection is devoted to the proof of Theorem \ref{theorem:ExtremalSystemClassification}. For $z\in \bdy \bb R_+^n$ and $\lambda>0$ define
\begin{equation*}
	y^{z, \lambda} = z + \frac{\lambda^2(y - z)}{\abs{ y - z}^2}
	\qquad
	\text{ and }
	\qquad
	x^{z, \lambda} = z + \frac{\lambda^2(x - z)}{\abs{x - z}^2}
\end{equation*}	
for $	y\in \bdy\bb R_+^n\setminus\{z\}$ and $x\in \bb R_+^n$ respectively. For any such $x, y$ and $z$ we have
\begin{equation}
\label{eq:DistanceBetweenInvertedPoints}
	|x^{z, \lambda}- y^{z, \lambda}|
	= 
	\frac{\lambda}{\abs{x - z}}\cdot\frac{\lambda}{\abs{y - z}}\abs{x - y}
\end{equation} 
and 
\begin{equation}
\label{eq:SphereInversionSimilarTriangles}
	\abs{y - z}|x - y^{z, \lambda}|
	= 
	\abs{x - z}|x^{z, \lambda} - y|
\end{equation}
as well as 
\begin{equation}
\label{eq:InversionNthCoordinate}
	\left(x^{z, \lambda}\right)_n 
	= 
	\left(\frac{\lambda}{\abs{x - z}}\right)^2 x_n. 
\end{equation}
Define the Kelvin-type transformations
\begin{equation*}
	u_{z, \lambda}(y)
	= 
	\left(\frac{\lambda}{\abs{ y - z}}\right)^{n - \alpha} u(y^{z, \lambda})
	\qquad
	\text{ for } y\in \bdy \bb R_+^n \setminus\{z\}
\end{equation*}
and
\begin{equation*}
	v_{z, \lambda}(x)
	= 
	\left(\frac{\lambda}{\abs{x - z}}\right)^{n - \alpha - 2\beta}
	v(x^{z, \lambda})
	\qquad
	\text{ for } x\in \bb R_+^n. 
\end{equation*}
For each $z\in \bdy \bb R_+^n$ and $\lambda >0$ these functions satisfy $\norm{u}_{L^{\theta +1}(\bdy \bb R_+^n)} = \norm{u_{z, \lambda}}_{L^{\theta + 1}(\bdy \bb R_+^n)}$ and $\norm{v}_{L^{\kappa + 1}(\bb R_+^n)} = \norm{v_{z, \lambda}}_{L^{\kappa + 1}(\bb R_+^n)}$. 
\begin{lemma}
If $u\in L^{\theta + 1}(\bdy \bb R_+^n)$ and $v\in L^{\kappa+ 1}(\bb R_+^n)$ are nonnegative functions satisfying \eqref{eq:ConformallyInvariantExtremalSystem} then for every $z\in \bb R_+^n$ and every $\lambda>0$
\begin{equation}
\label{eq:KelvinTransformedConformallyInvariantExtremalSystem}
\begin{cases}
	\displaystyle
	u_{z, \lambda}(y) 
	= 
	\int_{\bb R_+^n}\frac{x_n^\beta v_{z, \lambda}^\kappa(x)}{\abs{x - y}^{n - \alpha}} \; \d x
	& 
	\text{ for } y\in \bdy \bb R_+^n\setminus\{z\}
	\\
	\displaystyle
	v_{z, \lambda}(x) 
	=
	\int_{\bdy \bb R_+^n}\frac{x_n^\beta u_{z, \lambda}^\theta (y)}{\abs{x - y}^{n - \alpha}} \; \d y
	&
	\text{ for } x\in \bb R_+^n. 
\end{cases}
\end{equation}
\end{lemma}
\begin{proof}
Using the change of variable $y\mapsto y^{z, \lambda}$ in the second of equations \eqref{eq:ConformallyInvariantExtremalSystem} gives
\begin{eqnarray*}
	v(x)
	& = & 
	\int_{\bdy \bb R_+^n} \frac{x_n^\beta u^\theta(y^{z, \lambda})}{\abs{x - y^{z, \lambda}}^{n - \alpha}}
	\left(\frac{\lambda}{\abs{y - z}}\right)^{2(n - 1)}\; \d y
	\\
	& = & 
	\int_{\bdy \bb R_+^n} \frac{x_n^\beta u_{z, \lambda}^\theta(y^{z, \lambda})}{\abs{x - y^{z, \lambda}}^{n - \alpha}}
	\left(\frac{\lambda}{\abs{y - z}}\right)^{n - \alpha}\; \d y. 
\end{eqnarray*}
Therefore, using \eqref{eq:DistanceBetweenInvertedPoints} and \eqref{eq:InversionNthCoordinate} we obtain
\begin{eqnarray*}
	v_{z, \lambda}(x)
	& = & 
	\left(\frac{\lambda}{\abs{x - z}}\right)^{n -\alpha}
	\int_{\bdy \bb R_+^n}\frac{x_n^\beta u_{z, \lambda}^\theta(y)}{\abs{x^{z, \lambda}- y^{z, \lambda}}^{n - \alpha}}
	 \left(\frac{\lambda}{\abs{y - z}}\right)^{n - \alpha}\; \d y
	\\
	& = & 
	\int_{\bdy \bb R_+^n}\frac{x_n^\beta u_{z, \lambda}^\theta(y) }{\abs{x- y}^{n - \alpha}}\; \d y. 
\end{eqnarray*}
The proof of the first equation in \eqref{eq:ConformallyInvariantExtremalSystem} is similar. 
\end{proof}
\begin{lemma}
If $u, v$ satisfy \eqref{eq:ConformallyInvariantExtremalSystem} then for all $z\in \bdy \bb R_+^n$ and $\lambda>0$, 
\begin{equation}
\label{eq:ConformallyInvariantKelvinu-u}
\begin{array}{lcl}
	\multicolumn{3}{l}{
	u_{z, \lambda}(y)- u(y)
	}
	\\
	& = & 
	\displaystyle
	\int_{\bb R_+^n\setminus B_\lambda^+(z)}
	\left(\frac{x_n^\beta}{\abs{x - y}^{n - \alpha}} 
	- 
	\left(\frac{\lambda}{\abs{x - z}}\right)^{n - \alpha}\frac{x_n^\beta}{\abs{x^{z, \lambda} - y}^{n - \alpha}}
	\right)
	\left(v_{z, \lambda}^\kappa(x) - v^\kappa(x)\right)\; \d x
\end{array}
\end{equation}
for $y\in \bdy \bb R_+^n\setminus \{z\}$ and 
\begin{equation}
\label{eq:ConformallyInvariantKelvinv-v}
\begin{array}{lcl}
	\multicolumn{3}{l}{
	v_{z, \lambda}(x) - v(x)
	}
	\\
	& = & 
	\displaystyle
	\int_{\bdy\bb R_+^n\setminus B_\lambda^{n - 1}(z)}
	\left(\frac{x_n^\beta}{\abs{x - y}^{n - \alpha}} 
	- 
	\left(\frac{\lambda}{\abs{x - z}}\right)^{n - \alpha}\frac{x_n^\beta}{\abs{x^{z, \lambda} - y}^{n - \alpha}}
	\right)
	\left(u_{z, \lambda}^\theta(y) - u^\theta(y)\right)\; \d y
\end{array}
\end{equation}
for $x\in \bb R_+^n$. Moreover, for any $z\in \bdy \bb R_+^n$ and $\lambda>0$, 
\begin{equation}
\label{eq:KernelDifferenceInequality}
	\frac{1}{\abs{x - y}^{n - \alpha}}
	>
	\left(\frac{\lambda}{\abs{x - z}}\right)^{n - \alpha} \frac{1}{\abs{x^{z, \lambda} - y}^{n - \alpha}}
	\qquad
	\text{ for }
	x\in \bb R_+^n\setminus B_\lambda(z), y\in \bdy \bb R_+^n \setminus B_\lambda(z). 
\end{equation}
\end{lemma}
\begin{proof}
To prove equation \eqref{eq:ConformallyInvariantKelvinv-v} use the change of variable $y\mapsto y^{z, \lambda}$ to get
\begin{equation*}
	\int_{B_\lambda^{n- 1}(z)}\frac{x_n^\beta}{\abs{x - y}^{n - \alpha}} u^\theta(y)\; \d y
	= 
	\int_{\bdy \bb R_+^n\setminus B_\lambda^{n - 1}(z)}
	\frac{x_n^\beta}{\abs{x - y^{z, \lambda}}^{n - \alpha}}u^\theta(y^{z, \lambda})
	\left(\frac{\lambda}{\abs{y - z}}\right)^{2(n - 1)}\; \d y. 
\end{equation*}
Now use \eqref{eq:SphereInversionSimilarTriangles} to obtain
\begin{equation}
\label{eq:ConformallyInvariantExtremalv}
\begin{array}{lcl}
	\multicolumn{3}{l}{
	v(x)
	}
	\\
	& = & 
	\displaystyle
	\int_{\bdy \bb R_+^n\setminus B_\lambda^{n-1}(z)}
	\frac{x_n^\beta u^\theta(y)}{\abs{x - y}^{n - \alpha}}\; \d y
	+
	\int_{\bdy \bb R_+^n\setminus B_\lambda^{n-1}(z)}
	\left(\frac{\lambda}{\abs{y  - z}}\right)^{n - \alpha}
	\frac{x_n^\beta u_{z, \lambda}^\theta(y)}{\abs{x - y^{z, \lambda}}^{n - \alpha}}\; \d y
	\\
	& = & 
	\displaystyle
	\int_{\bdy \bb R_+^n\setminus B_\lambda^{n-1}(z)}
	\frac{x_n^\beta u^\theta(y)}{\abs{x - y}^{n - \alpha}}\; \d y
	+
	\int_{\bdy \bb R_+^n\setminus B_\lambda^{n-1}(z)}
	\left(\frac{\lambda}{\abs{x  - z}}\right)^{n - \alpha}
	\frac{x_n^\beta u_{z, \lambda}^\theta(y)}{\abs{x^{z, \lambda} - y}^{n - \alpha}}\; \d y. 
\end{array}
\end{equation}
Evaluating this expression at $x^{z, \lambda}$ and using \eqref{eq:InversionNthCoordinate} gives
\begin{equation}
\label{eq:ConformallyInvariantExtremalvKelvin}
\begin{array}{lcl}
	\multicolumn{3}{l}{
	v_{z, \lambda}(x)
	}
	\\
	& = & 
	\displaystyle
	\int_{\bdy \bb R_+^n\setminus B_\lambda^{n-1}(z)}
	\left(\frac{\lambda}{\abs{x - z}}\right)^{n - \alpha}
	\frac{x_n^\beta u^\theta(y)}{\abs{x^{z, \lambda} - y}^{n - \alpha}}\; \d y
	\\
	& & 
	\displaystyle
	+ 
	\int_{\bdy \bb R_+^n\setminus B_\lambda^{n-1}(z)}
	\left(\frac{\lambda}{\abs{x - z}}\right)^{n - \alpha}\left(\frac{\lambda}{\abs{y - z}}\right)^{n - \alpha}
	\frac{x_n^\beta u_{z, \lambda}^\theta(y)}{\abs{x^{z, \lambda} - y^{z, \lambda}}^{n - \alpha}}\; \d y
	\\
	& = & 
	\displaystyle
	\int_{\bdy \bb R_+^n\setminus B_\lambda^{n-1}(z)}
	\left(\frac{\lambda}{\abs{x - z}}\right)^{n - \alpha}
	\frac{x_n^\beta u^\theta(y)}{\abs{x^{z, \lambda} - y}^{n - \alpha}}\; \d y
	+ 
	\int_{\bdy \bb R_+^n\setminus B_\lambda^{n-1}(z)}
	\frac{x_n^\beta u_{z, \lambda}^\theta(y)}{\abs{x - y}^{n - \alpha}}\; \d y, 
\end{array}
\end{equation}
where \eqref{eq:DistanceBetweenInvertedPoints} was used in the final equality. Equation \eqref{eq:ConformallyInvariantKelvinv-v} now follows from equations \eqref{eq:ConformallyInvariantExtremalv} and \eqref{eq:ConformallyInvariantExtremalvKelvin}. Equation \eqref{eq:ConformallyInvariantKelvinu-u} follows from a similar computation. 

Finally,  observe that inequality \eqref{eq:KernelDifferenceInequality} is equivalent to the inequality $h_\lambda(\abs{x- z}, \abs{y - z})>0$ for $\abs{x- z}> \lambda$ and $\abs{y - z}> \lambda$, where $h_\lambda(a, b) = \lambda^4 + a^2b^2 -\lambda^2(a^2 + b^2)$ for $(a, b)\in \bb R^2$. Evidently both $\partial_ah_\lambda$ and $\partial_bh_\lambda$ are strictly positive on $(\lambda, \infty)\times(\lambda, \infty)$ so since $h_\lambda(\lambda, \lambda) = 0$, inequality \eqref{eq:KernelDifferenceInequality} holds. 
\end{proof}
For $z\in \bdy \bb R_+^n$ and $\lambda>0$ define
\begin{eqnarray*}
	\mathcal B^{n -1}(z, \lambda)
	& = & 
	\{y\in \bdy \bb R_+^n\setminus B_\lambda^{n-1}(z): u(y)< u_{z, \lambda}(y)\}
	\\
	\mathcal B^n(z, \lambda)
	& = & 
	\{x\in \bb R_+^n\setminus B_\lambda^+(z): v(x)< v_{z, \lambda}(x)\}. 
\end{eqnarray*}
We begin by showing that if $\lambda$ is sufficiently small the both of $\mathcal B^{n - 1}(z, \lambda)$ and $\mathcal B^n(z, \lambda)$ have zero measure. 
\begin{lemma}
\label{lemma:ConformallyInvariantMovingSpheresCanStart}
Under the hypotheses of Theorem \ref{theorem:ExtremalSystemClassification}, for all $z\in \bdy \bb R_+^n$ there exists $\lambda(z)>0$ such that if $0< \mu\leq \lambda(z)$ then both 
\begin{equation}
\label{eq:uDominatesKelvinTransform}
	u(y) \geq u_{z, \mu}(y)
	\qquad
	\text{ for a.e. } y\in \bdy \bb R_+^n \setminus B_\mu^{n-1}(z)
\end{equation}
and
\begin{equation}
\label{eq:vDominatesKelvinTransform}
	v(x) \geq v_{z, \mu}(x)
	\qquad
	\text{ for a.e. } x\in \bb R_+^n\setminus B_\mu^+(z). 
\end{equation}
\end{lemma}
\begin{proof}
Let $z\in \bdy \bb R_+^n$. For $y\in \mathcal B(z,\lambda)^{n - 1}$, using equation \eqref{eq:ConformallyInvariantKelvinu-u}, inequality \eqref{eq:KernelDifferenceInequality} and the Mean Value Theorem gives
\begin{eqnarray*}
	0 
	& < &
	u_{z, \lambda}(y) - u(y)
	\\
	& \leq & 
	\int_{\mathcal B^n(z, \lambda)}
	\left(
	\frac{x_n^\beta}{\abs{x - y}^{n - \alpha} }
	-
	\left(\frac{\lambda}{\abs{x - z}}\right)^{n - \alpha}\frac{x_n^\beta}{\abs{x^{z, \lambda}- y}^{n - \alpha}}
	\right)
	\left(v_{z, \lambda}^\kappa(x) - v^\kappa(x)\right)\; \d x
	\\
	& \leq & 
	\int_{\mathcal B^n(z, \lambda)}
	\frac{x_n^\beta}{\abs{x - y}^{n - \alpha}}\left(v_{z, \lambda}^\kappa(x) - v^\kappa(x)\right)\; \d x
	\\
	& \leq & 
	\kappa \int_{\mathcal B^n(z, \lambda)}
	\frac{x_n^\beta}{\abs{x - y}^{n - \alpha}}v_{z, \lambda}^{\kappa - 1}(x)
	\left(v_{z, \lambda}(x) - v(x)\right)\; \d x. 
\end{eqnarray*}
Applying Corollary \ref{coro:ConformallyInvariantSharpHalfSpaceInequality} and H\"older's inequality gives
\begin{equation}
\label{eq:Kelvinu-uFirstLebesgueEstimate}
\begin{array}{lcl}
\multicolumn{3}{l}{
	\displaystyle
	\norm{u_{z, \lambda} - u}_{L^{\theta + 1}(\mathcal B^{n - 1}(z, \lambda))}
	}
	\\
	& \leq&
	\displaystyle
	C
	\norm{v_{z, \lambda}^{\kappa - 1}\left(v_{z, \lambda} - v\right)}_{L^{(\kappa + 1)/\kappa}(\mathcal B^n(z, \lambda))}
	\\
	& \leq & 
	\displaystyle
	C
	\norm{v_{z, \lambda}}_{L^{\kappa + 1}(\mathcal B^n(z, \lambda))}^{\kappa - 1}
	\norm{v_{z, \lambda} - v}_{L^{\kappa + 1}(\mathcal B^n(z, \lambda))}
	\\
	& = & 
	\displaystyle
	C
	\norm{v}_{L^{\kappa + 1}(\mathcal B(z, \lambda)^{z,\lambda})}^{\kappa - 1}
	\norm{v_{z, \lambda} - v}_{L^{\kappa + 1}(\mathcal B^n(z, \lambda))}
\end{array}
\end{equation}
for some constant $C = C(n, \alpha, \beta)>0$, where $\mathcal B^n(z, \lambda)^{z, \lambda} = \{x^{z, \lambda}: x\in \mathcal B^n(z, \lambda)\}\subset B_\lambda^+(z)$. We estimate $\norm{v_{z, \lambda} - v}_{L^{\kappa + 1}(\mathcal B^n(z, \lambda))}$ separately in two cases depending on whether $\alpha \geq 1$ or $\alpha <1$. \\
{\bf Case 1:} Assume $\alpha \geq 1$ (and hence $\theta\geq1$). For $x\in \mathcal B^n(z, \lambda)$, equation \eqref{eq:ConformallyInvariantKelvinv-v} and inequality \eqref{eq:KernelDifferenceInequality} give
\begin{eqnarray}
\label{eq:IronManWatch}
	0
	& < &
	v_{z, \lambda}(x) - v(x)
	\notag
	\\
	&\leq & 
	\int_{\mathcal B^{n-1}(z, \lambda)}
	\left(
	\frac{x_n^\beta}{\abs{x - y}^{n - \alpha} }
	-
	\left(\frac{\lambda}{\abs{x - z}}\right)^{n - \alpha}\frac{x_n^\beta}{\abs{x^{z, \lambda}- y}^{n - \alpha}}
	\right)
	\left(u_{z, \lambda}^\theta(y) - u^\theta(y)\right)\; \d y
	\notag
	\\
	&\leq & 
	\int_{\mathcal B^{n - 1}(z, \lambda)}
	\frac{x_n^\beta}{\abs{x - y}^{n - \alpha}}\left(u_{z, \lambda}^\theta(y) - u^\theta(y)\right)\; \d y. 
\end{eqnarray}
The Mean Value Theorem now gives 
\begin{equation*}
	0< v_{z, \lambda}(x) - v(x)
	\leq
	\theta \int_{\mathcal B^{n - 1}(z, \lambda)}
	\frac{x_n^\beta}{\abs{x - y}^{n - \alpha}}u_{z, \lambda}^{\theta - 1}(y)
	\left(u_{z, \lambda}(y) - u(y)\right)\; \d y
\end{equation*}
for a.e. $x\in \mathcal B^n(z, \lambda)$. Applying Theorem \ref{theorem:ConformallyInvariantSharpHalfSpaceInequality} and H\"older's inequality gives
\begin{equation}
\label{eq:Kelvinv-vFirstLebesgueEstimate}
\begin{array}{lcl}
	\multicolumn{3}{l}{
	\displaystyle
	\norm{v_{z, \lambda} - v}_{L^{\kappa + 1}(\mathcal B^n(z, \lambda))}
	}
	\\
	& \leq & 
	\displaystyle
	C
	\norm{u_{z, \lambda}^{\theta - 1}\left(u_{z, \lambda} - u\right)}_{L^{(\theta + 1)/\theta}(\mathcal B^{n - 1}(z, \lambda))}
	\\
	& \leq & 
	\displaystyle
	C
	\norm{u_{z, \lambda}}_{L^{\theta + 1}(\mathcal B^{n - 1}(z, \lambda))}^{\theta - 1}
	\norm{u_{z, \lambda} - u}_{L^{\theta+ 1}(\mathcal B^{n - 1}(z, \lambda))}
	\\
	& = & 
	\displaystyle
	C
	\norm{u}_{L^{\theta + 1}(\mathcal B^{n -1}(z, \lambda)^{z, \lambda})}^{\theta - 1}
	\norm{u_{z, \lambda} - u}_{L^{\theta+ 1}(\mathcal B^{n - 1}(z, \lambda))}
\end{array}
\end{equation}
for some positive constant $C = C(n, \alpha, \beta)$, where $\mathcal B^{n - 1}(z, \lambda)^{z, \lambda} = \{y^{z, \lambda}: y\in \mathcal B^{n - 1}(z, \lambda)\}\subset B_\lambda^{n - 1}(z)$. Using \eqref{eq:Kelvinv-vFirstLebesgueEstimate} in \eqref{eq:Kelvinu-uFirstLebesgueEstimate} gives
\begin{equation}
\label{eq:Kelvinu-uSecondLebesgueEstimate}
\begin{array}{lcl}
	\multicolumn{3}{l}{
	\displaystyle
	\norm{u_{z, \lambda} - u}_{L^{\theta + 1}(\mathcal B^{n - 1}(z, \lambda))}
	}
	\\
	&\leq & 
	\displaystyle
	C\norm{u}_{L^{\theta + 1}(\mathcal B^{n - 1}(z, \lambda)^{z, \lambda})}^{\theta - 1}
	\norm{v}_{L^{\kappa + 1}(\mathcal B^n(z, \lambda)^{z, \lambda})}^{\kappa -1}
	\norm{u_{z, \lambda} - u}_{L^{\theta + 1}(\mathcal B^{n - 1}(z, \lambda))}. 
\end{array}
\end{equation}
Similarly, using \eqref{eq:Kelvinu-uFirstLebesgueEstimate} in \eqref{eq:Kelvinv-vFirstLebesgueEstimate} gives
\begin{equation}
\label{eq:Kelvinv-vSecondLebesgueEstimate}
\begin{array}{lcl}
	\multicolumn{3}{l}{
	\displaystyle
	\norm{v_{z, \lambda} - v}_{L^{\kappa + 1}(\mathcal B^n(z, \lambda))}
	}
	\\
	&\leq & 
	C\norm{u}_{L^{\theta + 1}(\mathcal B^{n - 1}(z, \lambda)^{z, \lambda})}^{\theta - 1}
	\norm{v}_{L^{\kappa + 1}(\mathcal B^n(z, \lambda)^{z, \lambda})}^{\kappa -1}
	\norm{v_{z, \lambda} - v}_{L^{\kappa + 1}(\mathcal B^n(z, \lambda))}. 
\end{array}
\end{equation}
Since $u\in L^{\theta + 1}(\bdy \bb R_+^n)$ and $v\in L^{\kappa + 1}(\bb R_+^n)$ and since $\theta\geq 1$ there is $\lambda_0(z)>0$ such that if $0< \lambda< \lambda_0(z)$ then 
\begin{equation*}
	C\norm{u}_{L^{\theta + 1}(B_\lambda^{n-1}(z))}^{\theta - 1}
	\norm{v}_{L^{\kappa + 1}(B_\lambda^+(z))}^{\kappa -1}
	\leq 
	\frac 12. 
\end{equation*}
For any such $\lambda$ estimates \eqref{eq:Kelvinu-uSecondLebesgueEstimate} and \eqref{eq:Kelvinv-vSecondLebesgueEstimate} guarantee that 
\begin{equation*}
	\norm{u_{z, \lambda} - u}_{L^{\theta + 1}(\mathcal B^{n - 1}(z, \lambda))}
	\leq
	\frac 12
	\norm{u_{z, \lambda} - u}_{L^{\theta + 1}(\mathcal B^{n - 1}(z, \lambda))}
\end{equation*}
and
\begin{equation*}
	\norm{v_{z, \lambda} - v}_{L^{\kappa + 1}(\mathcal B^n(z, \lambda))}
	\leq
	\frac 12 
	\norm{v_{z, \lambda} - v}_{L^{\kappa + 1}(\mathcal B^n(z, \lambda))}
\end{equation*}
from which we deduce that both $\mathcal B^{n - 1}(z, \lambda)$ and $\mathcal B^n(z, \lambda)$ are measure-zero sets. \\
{\bf Case 2:} Assume $\alpha < 1$ (and hence $\theta < 1$). Let $r$ satisfy $1< \frac 1 \theta < r< \kappa$. For $x\in \mathcal B^n(z, \lambda)$ estimating as in \eqref{eq:IronManWatch} and using the Mean Value Theorem gives
\begin{eqnarray*}
	0
	& < &
	v_{z, \lambda}(x) - v(x)
	\\
	&\leq & 
	r\theta  \int_{\mathcal B^{n - 1}(z, \lambda)}
	\frac{x_n^\beta}{\abs{x - y}^{n - \alpha}}u_{z, \lambda}^{\theta - 1/r}(y)
	\left(u_{z, \lambda}^{1/r}(y) - u(y)^{1/r}\right)\; \d y. 
\end{eqnarray*}
Applying Theorem \ref{theorem:ConformallyInvariantSharpHalfSpaceInequality} and H\"older's inequality gives
\begin{equation}
\label{eq:Kelvinv-vFirstLebesgueEstimateTheta<1}
\begin{array}{lcl}
\multicolumn{3}{l}{
	\norm{v_{z, \lambda} - v}_{L^{\kappa + 1}(\mathcal B^n(z, \lambda))}
	}
	\\
	& \leq & 
	C
	\norm{u_{z, \lambda}^{\theta - 1/r}\left(u_{z, \lambda}^{1/r} - u^{1/r}\right)}_{L^p(\mathcal B^{n - 1}(z, \lambda))}
	\\
	& \leq & 
	C
	\norm{u_{z, \lambda}}_{L^{\theta + 1}(\mathcal B^{n - 1}(z, \lambda)}^{\theta - 1/r}
	\norm{u_{z, \lambda}^{1/r} - u^{1/r}}_{L^{\theta+ 1}(\mathcal B^{n - 1}(z, \lambda))}
\end{array}
\end{equation}
for some positive constant $C = C(n, \alpha, \beta)$. To estimate the norm of $u_{z, \lambda}^{1/r} - u^{1/r}$ we define $\mathcal G^n(z, \lambda) = \bb R_+^n\setminus(B_\lambda^+(z)\cup\mathcal B^n(z, \lambda))$ (the ``good'' set) and 
\begin{equation*}
	Q(x,y) 
	= 
	\frac{x_n^\beta v^\kappa(x)}{\abs{x - y}^{n - \alpha}}
	+ 
	\left(\frac{\lambda}{\abs{x - z}}\right)^{n - \alpha}
	\frac{x_n^\beta v_{z, \lambda}^\kappa(x)}{\abs{x^{z, \lambda}- y}^{n - \alpha}}. 
\end{equation*}
Define also
\begin{eqnarray*}
	a(y) & = & \int_{\mathcal B^n(z, \lambda)}Q(x,y)\; \d x\\
	c(y) &  = & \int_{\mathcal G(z, \lambda)}Q(x,y)\; \d x
\end{eqnarray*}
so that $u(y) = a(y) + c(y)$. Equations \eqref{eq:DistanceBetweenInvertedPoints} and \eqref{eq:SphereInversionSimilarTriangles} give
\begin{equation*}
\begin{array}{lcl}
	\multicolumn{3}{l}{
	\displaystyle
	Q(x,y) - \left(\frac\lambda{\abs{y - z}}\right)^{n - \alpha} Q(x, y^{z, \lambda})
	}
	\\
	& = & 
	\displaystyle
	\left(\frac{x_n^\beta}{\abs{x - y}^{n - \alpha}} - \left(\frac\lambda{\abs{x - z}}\right)^{n - \alpha}\frac{x_n^\beta}{\abs{x^{z, \lambda} - y}^{n - \alpha}}\right)
	\left(v^\kappa(x) - v_{z, \lambda}^\kappa(x)\right)
\end{array}
\end{equation*}
for a.e. $x\in \bb R_+^n\setminus B_\lambda^+(z)$, $y\in \bdy \bb R_+^n\setminus B_\lambda^{n - 1}(z)$. In particular, for $y\in \bdy \bb R_+^n\setminus B_\lambda^{n - 1}(z)$
\begin{equation*}
	a(y)< a_{z, \lambda}(y) := \left(\frac{\lambda}{\abs{y - z}}\right)^{n - \alpha}a(y^{z, \lambda})
\end{equation*}
and 
\begin{equation*}
	c(y) > c_{z, \lambda}(y) := \left(\frac{\lambda}{\abs{y - z}}\right)^{n - \alpha} c(y^{z, \lambda}). 
\end{equation*}
If $\overline c$ is any nonnegative function satisfying $c_{z, \lambda} \leq \overline c \leq c$ in $\bdy \bb R_+^n\setminus B_{\lambda}^{n - 1}(z)$ then for a.e. $y\in \mathcal B^{n - 1}(z, \lambda)$
\begin{eqnarray}
\label{eq:1/rPowerFunctionConcave}
	0 
	& < & 
	u_{z, \lambda}^{1/r}(y) - u^{1/r}(y)
	\notag
	\\
	&\leq & 
	\left(a_{z, \lambda}(y) + \overline c(y)\right)^{1/r} - \left(a(y) + \overline c(y)\right)^{1/r}
	\notag
	\\
	&\leq & 
	a_{z, \lambda}(y)^{1/r} - a(y)^{1/r}. 
\end{eqnarray}
Using equations \eqref{eq:DistanceBetweenInvertedPoints} and \eqref{eq:SphereInversionSimilarTriangles} once more we get
\begin{equation*}
	a_{z, \lambda}(y)
	= 
	\int_{\mathcal B^n(z, \lambda)}\abs{\left(\varphi_1(x,y), \varphi_2(x,y)\right)}_{\ell^r}^r\; \d x, 
\end{equation*}
where
\begin{equation*}
	\varphi_1(x,y) 
	= 
	\left( \frac{x_n^\beta v_{z, \lambda}^\kappa(x)}{\abs{x - y}^{n - \alpha}}\right)^{1/r}
	\quad
	\text{ and } 
	\quad
	\varphi_2(x,y) 
	= 
	\left( \left(\frac{\lambda}{\abs{x - z}}\right)^{n - \alpha}\frac{x_n^\beta v^\kappa(x)}{\abs{x^{z, \lambda} - y}^{n - \alpha}}\right)^{1/r}. 
\end{equation*}
Setting also
\begin{equation*}
	\psi_1(x,y)
	= 
	\left(\frac{x_n^\beta v^\kappa(x)}{\abs{x- y}^{n - \alpha}}\right)^{1/r}
	\quad
	\text{ and }
	\quad
	\psi_2(x,y)
	= 
	\left(\left(\frac{\lambda}{\abs{x- z}}\right)^{n - \alpha}\frac{x_n^\beta v_{z, \lambda}^\kappa(x)}{\abs{x^{z, \lambda} - y}^{n - \alpha}} \right)^{1/r}
\end{equation*}
gives
\begin{equation*}
	a(y) = \int_{\mathcal B^n(z, \lambda)}\abs{\left(\psi_1(x,y), \psi_2(x,y)\right)}_{\ell^r}^r\; \d x. 
\end{equation*}
Returning to equation \eqref{eq:1/rPowerFunctionConcave} and using the Mean Value Theorem, for a.e. $y\in \mathcal B^{n - 1}(z, \lambda)$ gives
\begin{eqnarray*}
	0
	& < & 
	u_{z, \lambda}^{1/r}(y) - u^{1/r}(y)
	\\
	&\leq & 
	\left[\int_{\mathcal B^n(z, \lambda)}\abs{\left(\varphi_1(x,y), \varphi_2(x,y)\right)}_{\ell^r}^r\; \d x \right]^{1/r}
	-
	\left[\int_{\mathcal B^n(z, \lambda)}\abs{\left(\psi_1(x,y), \psi_2(x,y)\right)}_{\ell^r}^r\; \d x \right]^{1/r}
	\\
	&\leq & 
	\left[ \int_{\mathcal B^n(z, \lambda)} \abs{\varphi_1 - \psi_1}^r + \abs{\varphi_2 - \psi_2}^r\; \d x\right]^{1/r}
	\\
	& \leq & 
	C\left[ \int_{\mathcal B^n(z, \lambda)} \frac{x_n^\beta}{\abs{x - y}^{n - \alpha}} \left(v_{z, \lambda}^{\kappa/r}(x) - v^{\kappa/r}(x)\right)^r\; \d x\right]^{1/r}
	\\
	& \leq & 
	C\left[ \int_{\mathcal B^n(z, \lambda)} \frac{x_n^\beta}{\abs{x - y}^{n - \alpha}}v_{z, \lambda}^{\kappa - r}(x) \left(v_{z, \lambda}(x) - v(x)\right)^r\; \d x\right]^{1/r}.
\end{eqnarray*}
Applying Corollary \ref{coro:ConformallyInvariantSharpHalfSpaceInequality} gives
\begin{eqnarray*}
	\norm{u_{z, \lambda}^{1/r} - u^{1/r}}_{L^{r(\theta + 1)}(\mathcal B^{n - 1}(z, \lambda))}
	&\leq & 
	C\norm{v_{z, \lambda}^{\kappa - r}(v_{z, \lambda} - v)^r}_{L^{(\kappa + 1)/\kappa}(\mathcal B^n(z, \lambda))}^{1/r}
	\\
	&\leq & 
	C\norm{v_{z, \lambda}}_{L^{\kappa + 1}(\mathcal B^n(z, \lambda))}^{\frac \kappa r - 1} \norm{v_{z, \lambda} - v}_{L^{\kappa + 1}(\mathcal B^n(z, \lambda))}. 
\end{eqnarray*}
Finally, using this estimate in \eqref{eq:Kelvinv-vFirstLebesgueEstimateTheta<1} gives
\begin{equation}
\label{eq:vLambda-vFinalEstimateTheta<1}
\begin{array}{lcl}
	\multicolumn{3}{l}{
	\norm{v_{z, \lambda} - v}_{L^{\kappa + 1}(\mathcal B^n(z, \lambda))}
	}
	\\
	& \leq & 
	\displaystyle
	C\norm{u_{z, \lambda}}_{L^{\theta + 1}(\mathcal B^{n -1}(z, \lambda))}^{\theta - 1/r}
	\norm{v_{z, \lambda}}_{L^{\kappa + 1}(\mathcal B^n(z, \lambda))}^{\frac \kappa r - 1}
	\norm{v_{z, \lambda} - v}_{L^{\kappa + 1}(\mathcal B^n(z, \lambda))}
	\\
	& = & 
	\displaystyle
	C\norm{u}_{L^{\theta + 1}(\mathcal B^{n-1}(z, \lambda)^{z, \lambda})}^{\theta - 1/r}
	\norm{v}_{L^{\kappa + 1}(\mathcal B^n(z, \lambda)^{z, \lambda})}^{\frac \kappa r - 1}
	\norm{v_{z, \lambda} - v}_{L^{\kappa + 1}(\mathcal B^n(z, \lambda))}. 
\end{array}
\end{equation}
This estimate guarantees the existence of $\lambda_0(z)>0$ such that $\abs{\mathcal B^n(z, \lambda)} = 0$ whenever $0< \lambda< \lambda_0(z)$. Estimate \eqref{eq:Kelvinu-uFirstLebesgueEstimate} now guarantees that $\abs{\mathcal B^{n - 1}(z, \lambda)} = 0$ when $0< \lambda< \lambda_0(z)$. 
\end{proof}

In view of Lemma \ref{lemma:ConformallyInvariantMovingSpheresCanStart}, for each $z\in \bdy \bb R_+^n$ the quantity 
\begin{equation*}
	\bar\lambda(z)
	= 
	\sup\left\{
	\lambda>0: 
	\forall 0< \mu< \lambda, \text{ both \eqref{eq:uDominatesKelvinTransform} and \eqref{eq:vDominatesKelvinTransform} hold}
	\right\}
\end{equation*}
is well-defined and positive. The next lemma shows that if the moving sphere process stops for some $z\in \bdy \bb R_+^n$  then $u$ is symmetric about $\bdy\bb R_+^n\cap \bdy B_{\bar\lambda(z)}$ and $v$ is symmetric about $\bdy B_{\bar\lambda(z)}\cap\bb R_+^n$. 
\begin{lemma}
\label{lemma:SymmetrySphere}
If $z\in\bdy \bb R_+^n$ satisfies $\bar\lambda(z)< \infty$ then both 
\begin{equation}
\label{eq:uSymmetries}
	u = u_{z, \bar\lambda(z)} \qquad \text{a.e in } \bdy \bb R_+^n\setminus \{z\}
\end{equation}
and
\begin{equation}
\label{eq:vSymmetries}
	v = v_{z, \bar\lambda(z)} \qquad \text{a.e. in } \bb R_+^n. 
\end{equation}
\end{lemma}
\begin{proof}
To establish the lemma it is sufficient to show that if $\bar\lambda(z)< \infty$ then both 
\begin{equation}
\label{eq:uSymmetricOutsideBall}
	u = u_{z, \bar\lambda(z)} \qquad \text{a.e. in } \bdy \bb R_+^n\setminus B_{\bar\lambda(z)}^{n-1}(z)
\end{equation}
and 
\begin{equation}
\label{eq:vSymmetricOutsideBall}
	v = v_{z, \bar\lambda(z)} \qquad \text{a.e. in } \bb R_+^n \setminus B_{\bar\lambda(z)}^+(z). 
\end{equation}
 The remainder of the proof is devoted to establishing equalities \eqref{eq:uSymmetricOutsideBall} and \eqref{eq:vSymmetricOutsideBall}. For simplicity these equalities are established only for $z = 0'$. The proof for general $z\in \bdy \bb R_+^n$ is similar. For ease of notation, we set $\bar\lambda = \bar\lambda(0')< \infty$, $y^{\bar\lambda} = y^{0', \bar\lambda}$, $u_{\bar\lambda} = u_{0', \bar\lambda}$, $\mathcal B^{n  -1}(\bar\lambda) = \mathcal B^{n - 1}(0', \bar\lambda)$. Similar notational conventions will be used for $x\in \bb R_+^n$ and $v$.   \\

First we show both that 
\begin{equation*}
	u \geq u_{\bar\lambda}
	\qquad
	\text{a.e. in } y\in \bdy \bb R_+^n\setminus B_{\bar\lambda}^{n-1}
\end{equation*}
and that 
\begin{equation}
\label{eq:vGEQvLambda}
	v\geq v_{\bar\lambda}
	\qquad
	\text{a.e. in } \bb R_+^n\setminus B_{\bar\lambda}^+. 
\end{equation}
\begin{claim}
\label{claim:MaxLambdaBadSetsEmpty}
Both of $\mathcal B^{n-1}(\bar \lambda)$ and $\mathcal B^n(\bar \lambda)$ are measure-zero sets. 
\end{claim}
\begin{proof}[Proof of Claim \ref{claim:MaxLambdaBadSetsEmpty}]
It suffices to show that $\abs{\mathcal B^{n-1}(\bar \lambda)} = 0$. Indeed, under this assumption for a.e. $x\in \bb R_+^n\setminus B_{\bar \lambda}$, equation \eqref{eq:ConformallyInvariantKelvinv-v} gives
\begin{equation*}
\begin{array}{lcl}
\multicolumn{3}{l}{
	v_{\bar\lambda}(x) - v(x)
	}
	\\
	& \leq & 
	\displaystyle
	\int_{\mathcal B^{n - 1}(\bar\lambda)}
	\left( \frac{x_n^\beta}{\abs{x - y}^{n -\alpha}} - \left(\frac{\bar \lambda}{\abs x}\right)^{n - \alpha}\frac{x_n^\beta}{|x^{\bar\lambda} - y|^{n -\alpha}}\right)
	\left(u_{\bar \lambda}^\theta(y) - u^\theta(y)\right)\; \d y
	\\
	& = & 
	0.
\end{array}
\end{equation*}
so that $\abs{\mathcal B^{n}(\bar \lambda)} = 0$. Now we show via proof by contradiction that $\abs{\mathcal B^{n-1}(\bar\lambda)} = 0$. For $0\leq \rho< r\leq +\infty$, $\lambda>0$ $\delta\geq 0$ and $M>0$ we use the notation 
\begin{eqnarray*}
	A(\rho, r) & =  & \{y\in \bdy \bb R_+^n: \rho < \abs y < r\}\\
	A(\rho, r; \lambda, \delta) & = & \{y\in A(\rho, r): u_\lambda(y) - u(y)> \delta\}\\
	A(\rho, r; \lambda, \delta; M) & = & \{y\in A(\rho, r;\lambda, \delta): u(y^{\bar\lambda})\leq M\}. 
\end{eqnarray*}
In this notation we have $\mathcal B^{n-1}(\bar\lambda) = A(\bar\lambda, \infty; \bar\lambda, 0)$. Suppose $\abs{\mathcal B^{n - 1}(\bar\lambda)}>0$. Choose $\delta>0$ (small) and $R>0$ (large) such that $3\abs{\mathcal B^{n - 1}(\bar\lambda)}\leq 4\abs{A(\bar\lambda, R; \bar\lambda, \delta)}$. Choose $M>0$ sufficiently large so that 
\begin{equation*}
	\abs{\{y\in A(\bar\lambda, R): u(y^{\bar\lambda})> M\}}
	\leq
	\abs{\{y\in \bdy \bb R_+^n: u(y)> M\}}
	\leq
	\frac{\abs{\mathcal B^{n - 1}(\bar\lambda)}}{4}. 
\end{equation*}
For such $M$ there holds
\begin{equation}
\label{eq:BadSetMeasureUpperBound}
	\abs{\mathcal B^{n-1}(\bar\lambda)}
	\leq
	2\abs{A(\bar\lambda, R; \bar\lambda, \delta; M)}. 
\end{equation}
Moreover, if $y\in A(\bar\lambda, R; \bar\lambda, \delta; M)$ then for any $0< \lambda< \bar\lambda$ 
\begin{eqnarray*}
	0
	&\geq & 
	u_\lambda(y) - u(y)
	\\
	&\geq & 
	u_\lambda(y) - u_{\bar\lambda}(y) + \delta. 
\end{eqnarray*}
For every $h\in C^0(\bdy \bb R_+^n)$ there exists $\lambda_h< \bar\lambda$ such that for all $\lambda\in (\lambda_h, \bar\lambda)$ and all $y\in A(\bar\lambda, R; \bar\lambda, \delta; M)$
\begin{eqnarray*}
	\delta
	& \leq & 
	u_{\bar\lambda}(y) - u_\lambda(y)
	\\
	&= & 
	\left(\left(\frac{\bar\lambda}{\abs y}\right)^{n - \alpha} - \left(\frac{\lambda}{\abs y}\right)^{n - \alpha}\right) u(y^{\bar\lambda})
	+
	\left(\frac{\lambda}{\abs y}\right)^{n - \alpha}\left( u(y^{\bar\lambda}) - u(y^\lambda)\right)
	\\
	&\leq & 
	M\bar\lambda^{\alpha - n}\left(\bar\lambda^{n - \alpha} - \lambda^{n - \alpha}\right)
	+ 
	|u(y^{\bar\lambda}) - h(y^{\bar\lambda})| + |h(y^{\bar\lambda}) - h(y^\lambda)| + |h(y^\lambda) - u(y^\lambda)|
	\\
	&\leq & 
	\frac\delta 2 +  |u(y^{\bar\lambda}) - h(y^{\bar\lambda})| + |h(y^\lambda) - u(y^\lambda)|.
\end{eqnarray*}
Combining this with \eqref{eq:BadSetMeasureUpperBound} we find that for any $h\in C^0\cap L^{\theta + 1}(\bdy \bb R_+^n)$, if $\lambda\in (\lambda_h, \bar\lambda)$ then 
\begin{eqnarray*}
	\abs{\mathcal B^{n - 1}(\bar\lambda)}
	&\leq & 
	2\abs{\{y\in A(\bar\lambda, R): |u(y^{\bar\lambda}) - h(y^{\bar\lambda})| \geq \delta/8\}}
	\\
	& & 
	+ 
	2\abs{\{y\in A(\bar\lambda, R): |u(y^{\lambda}) - h(y^{\lambda})| \geq \delta/8\}}
	\\
	&\leq & 
	4\abs{\{y\in\bdy \bb R_+^n: \abs{h(y) - u(y)}\geq \delta/8\}}
	\\
	&\leq & 
	C(n, \alpha)\delta^{-\theta - 1}\norm{h - u}_{L^{\theta+ 1}(\bdy \bb R_+^n)}^{\theta + 1}. 
\end{eqnarray*}
Choosing $h$ sufficiently close to $u$ in $L^{\theta+ 1}$-norm gives a contradiction. 
\end{proof}

Next we claim that exactly one of the following alternatives holds: 
\begin{enumerate}[(a)]
	\item $u> u_{\bar\lambda}$ a.e. in $\bdy \bb R_+^n\setminus B_{\bar\lambda}$ and $v> v_{\bar\lambda}$ a.e. in $\bb R_+^n\setminus B_{\bar\lambda}$. 
	\item $u= u_{\bar\lambda}$ a.e. in $\bdy \bb R_+^n\setminus B_{\bar\lambda}$ and $v= v_{\bar\lambda}$ a.e. in $\bb R_+^n\setminus B_{\bar\lambda}$.
\end{enumerate}
To see that the claim holds, suppose there is a positive-measure subset $A\subset \bdy \bb R_+^n\setminus B_{\bar\lambda}$ on which $u = u_{\bar\lambda}$. For every $y\in A$, equaiton \eqref{eq:ConformallyInvariantKelvinu-u} and estimates \eqref{eq:KernelDifferenceInequality} and \eqref{eq:vGEQvLambda} give
\begin{eqnarray*}
	0
	& = & 
	u_{\bar\lambda}(y) - u(y)
	\\
	& = & 
	\int_{\bb R_+^n\setminus B_{\bar\lambda}^+}
	\left(
	\frac{x_n^\beta}{\abs{x - y}^{n - \alpha}}
	-
	\left(\frac{\bar\lambda}{\abs{x}}\right)^{n - \alpha}
	\frac{x_n^\beta}{|x^{\bar\lambda} - y|^{n - \alpha}}
	\right)
	\left(v_{\bar\lambda}^\kappa(x) - v^\kappa(x)\right)\; \d x
	\\
	& \leq & 
	0 
\end{eqnarray*}
and consequently $v_{\bar\lambda} = v$ a.e. in $\bb R_+^n\setminus B_{\bar \lambda}^+$. Using this in \eqref{eq:ConformallyInvariantKelvinu-u} and in view of \eqref{eq:KernelDifferenceInequality} we find that $u = u_{\bar\lambda}$ a.e. in $\bdy\bb R_+^n\setminus B_{\bar\lambda}^{n-1}$. A similar computation shows that if there is a positive-measure subset of $\bb R_+^n\setminus B_{\bar\lambda}$ on which $v = v_{\bar\lambda}$ then alternative (b) holds. The claim is established. \\

To complete the proof of Lemma \ref{lemma:SymmetrySphere} we assume for the sake of obtaining a contradiction that $u> u_{\bar \lambda}$ a.e. in $\bdy \bb R_+^n\setminus \overline B_{\bar\lambda}$ and that $v> v_{\bar \lambda}$ a.e. in $\bb R_+^n\setminus \overline B_{\bar\lambda}$. We consider two cases depending on whether $\alpha>1$ or $\alpha \leq 1$. \\
If $\alpha \geq 1$ (and hence $\theta \geq 1$) then estimating as in the proof of Lemma \ref{lemma:ConformallyInvariantMovingSpheresCanStart} (see estimates \eqref{eq:Kelvinu-uSecondLebesgueEstimate}, \eqref{eq:Kelvinv-vSecondLebesgueEstimate}) gives a constant $C_0(n, \alpha, \beta)>0$ such that for all $\lambda>0$ both 
\begin{equation*}
\begin{array}{lcl}
\multicolumn{3}{l}{
	\displaystyle
	\norm{u_\lambda - u}_{L^{\theta + 1}(\mathcal B^{n-1}(\lambda))}
	}
	\\
	&\leq&
	\displaystyle
	C_0\norm{u}_{L^{\theta + 1}(\mathcal B^{n-1}(\lambda)^\lambda)}^{\theta - 1}\norm{v}_{L^{\kappa + 1}(\mathcal B^n(\lambda)^\lambda)}^{\kappa -1}
	\norm{u_\lambda - u}_{L^{\theta + 1}(\mathcal B^{n-1}(\lambda))}
\end{array}
\end{equation*}
and
\begin{equation*}
\begin{array}{lcl}
\multicolumn{3}{l}{
	\displaystyle
	\norm{v_\lambda - v}_{L^{\kappa + 1}(\mathcal B^n(\lambda))}
	}
	\\
	&\leq&
	\displaystyle
	C_0\norm{u}_{L^{\theta + 1}(\mathcal B^{n-1}(\lambda)^\lambda)}^{\theta - 1}\norm{v}_{L^{\kappa + 1}(\mathcal B^n(\lambda)^\lambda)}^{\kappa -1}
	\norm{v_\lambda - v}_{L^{\kappa + 1}(\mathcal B^n(\lambda))}. 
\end{array}
\end{equation*}	
If $\alpha \leq 1$ (and hence $\theta \leq 1$) then performing computations similar to those carried out in Case 2 of the proof of Lemma \ref{lemma:ConformallyInvariantMovingSpheresCanStart} (see estimate \eqref{eq:vLambda-vFinalEstimateTheta<1}) shows that for $1\leq \frac 1 \theta < r< \kappa$
\begin{equation*}
	\norm{v_\lambda- v}_{L^{\kappa + 1}(\mathcal B^n(\lambda))}
	\leq
	C\norm{u}_{L^{\theta + 1}(\mathcal B^{n - 1}(\lambda)^\lambda)}^{\theta - \frac 1 r}\norm{v}_{L^{\kappa + 1}(\mathcal B^n(\lambda)^\lambda)}^{\frac \kappa r - 1}
	\norm{v_\lambda- v}_{L^{\kappa + 1}(\mathcal B^n(\lambda))}. 
\end{equation*}
Based on these estimates, the desired contradiction will be obtained once we show that $\abs{\mathcal B^{n - 1}(\lambda)^\lambda}\to 0$ as $\lambda \to \bar\lambda^+$. First note that for any $0< \epsilon\leq \bar\lambda$ and any $R>\bar\lambda + \epsilon$, if $\bar\lambda\leq\lambda\leq\bar\lambda + \epsilon$ then we have 
\begin{equation*}
	\mathcal B^{n - 1}(\lambda)
	\subset A(\bar\lambda, \bar\lambda + \epsilon) \cup A(R, \infty)\cup A(\bar\lambda + \epsilon, R; \lambda, 0). 
\end{equation*}
Moreover, for such $\lambda$ each of the following three estimates hold
\begin{eqnarray*}
	\abs{A(\bar\lambda, \bar\lambda + \epsilon)^\lambda}
	&\leq & 
	\abs{A(\bar\lambda^2(\bar\lambda + \epsilon)^{-1}, (\bar\lambda + \epsilon)^2\bar\lambda^{-1})}
	\leq
	C(n)\bar\lambda^{n - 2}\epsilon,
	\\
	\abs{A(R, \infty)^\lambda}
	&\leq & 
	\abs{B^{n-1}(0, (\bar\lambda + \epsilon)^2R^{-1})}
	\leq
	C(n) \bar\lambda^{2(n-1)}R^{1-n}
\end{eqnarray*}
and 
\begin{eqnarray*}
	\abs{A(\bar\lambda+\epsilon, R; \lambda, 0)^\lambda}
	& = &
	\int_{A(\bar\lambda+\epsilon, R; \lambda, 0)^\lambda}\; \d y
	\\
	& = & 
	\int_{A(\bar\lambda+\epsilon, R; \lambda, 0)}\left(\frac{\lambda}{\abs y}\right)^{2(n - 1)}\; \d y
	\\
	& \leq & 
	\abs{A(\bar\lambda+\epsilon, R; \lambda, 0)}. 
\end{eqnarray*}
Thus, it suffices to show that for every $0< \epsilon< \bar\lambda$ and every $R> 2\bar\lambda$, the limit 
\begin{equation*}
	\lim_{\lambda\to \bar\lambda^+}\abs{A(\bar\lambda + \epsilon, R; \lambda, 0)}= 0
\end{equation*}
holds. We prove this by way of contradiction. Accordingly, suppose there is $0< \epsilon\leq \bar\lambda $, $R>2\bar\lambda$, a sequence $\lambda_k\to \bar\lambda^+$ and $\ell>0$ such that $\abs{A(\bar\lambda + \epsilon, R; \lambda_k, 0)}> \ell$ for all $k$. Observe first that there is $\delta>0$ depending on $n, \alpha, \beta, R, \epsilon$ and the distribution function of $v^\kappa - v_{\bar\lambda}^\kappa$ such that 
\begin{equation*}
	u(y) - u_{\bar\lambda}(y) \geq \delta
	\qquad
	\text{for all } y\in A(\bar\lambda + \epsilon, R). 
\end{equation*}
Now choose $M>0$ sufficiently large so that $2\abs{\{y\in \bdy \bb R_+^n: u(y)> M\}}< \ell$. Then we have $A(\bar\lambda + \epsilon, R; \lambda_k, 0)\subset A(\bar\lambda + \epsilon; R; \lambda_k, 0; M)\cup\{y\in \bdy \bb R_+^n: u(y)> M\}$ and consequently
\begin{equation}
\label{eq:Fingers}
	\ell < 2\abs{A(\bar\lambda + \epsilon, R; \lambda_k, 0; M)}. 
\end{equation}
For any $y\in A(\bar\lambda + \epsilon, R; \lambda_k, 0; M)$ and any $h\in C^0\cap L^{\theta+ 1}(\bdy \bb R_+^n)$ there is $k(h)\in \bb N$ such that for all $k> k(h)$
\begin{eqnarray*}
	\delta
	&\leq & 
	u(y) - u_{\bar\lambda}(y)
	\\
	& < & 
	u_{\lambda_k}(y) - u_{\bar\lambda(y)}
	\\
	&\leq & 
	\left(\left(\frac{\lambda_k}{\abs y}\right)^{n - \alpha} - \left(\frac{\bar\lambda}{\abs y}\right)^{n - \alpha}\right)u(y^{\bar\lambda})
	+ 
	\left(\frac{\lambda_k}{\abs y}\right)^{n - \alpha}\left(u(y^{\lambda_k}) - u(y^{\bar\lambda})\right)
	\\
	&\leq & 
	M(\bar\lambda + \epsilon)^{\alpha -n}\left(\lambda_k^{n - \alpha} - \bar\lambda^{n - \alpha}\right)
	+ 
	|u(y^{\lambda_k}) - h(y^{\lambda_k})| + |h(y^{\lambda_k}) - h(y^{\bar\lambda})| 
	\\
	&&+ |h(y^{\bar\lambda}) - u(y^{\bar\lambda})|
	\\
	&\leq & 
	\frac \delta 2 
	+
	|u(y^{\lambda_k}) - h(y^{\lambda_k})| + |h(y^{\bar\lambda}) - u(y^{\bar\lambda})|. 
\end{eqnarray*}
Returning to \eqref{eq:Fingers} we find that for $k> k(h)$, 
\begin{eqnarray*}
	\ell
	& < & 
	2\abs{\{y\in A(\bar\lambda + \epsilon, R): \delta\leq 4|u(y^{\lambda_k}) - h(y^{\lambda_k})|\}}
	\\
	&&
	+ 
	2\abs{\{y\in A(\bar\lambda+\epsilon, R): \delta\leq 4|u(y^{\bar\lambda}) - h(y^{\bar\lambda})|\}}
	\\
	&\leq & 
	4\abs{\{y\in \bdy \bb R_+^n: \delta\leq 4\abs{u(y)- h(y)}\}}
	\\
	&\leq & 
	C(n, \alpha)\delta^{-\theta - 1}\norm{u - h}_{L^{\theta + 1}(\bdy \bb R_+^n)}^{\theta + 1}. 
\end{eqnarray*}
Choosing $h\in C^0\cap L^{\theta + 1}(\bdy \bb R_+^n)$ sufficiently close to $u$ in $L^{\theta + 1}$-norm gives the desired contradiction. \\

\end{proof}
The following Theorem, due to Frank and Lieb \cite{FrankLieb2010} characterizes the reflection-invariant, absolutely continuous measures on $\bb R^n$. Their theorem holds in the absence of the absolute continuity assumption but since we only need to apply the theorem to the absolutely continuous measure $f(y)\; \d y = u^{\theta + 1}(y)\; \d y$ we choose not to state the theorem in its full generality. 
\begin{theorem}
\label{theorem:FrankLiebClassification}
If $f\in L^1(\bb R^n)$ is a nonnegative function satisfying both 
\begin{enumerate}[(a)]
	\item For every $z\in \bb R^n$ there exists $\lambda(z)>0$ and a set of full measure in $\bb R^n$ on which 
	\begin{equation*}
		f(x)
		= 
		\left(\frac{\lambda(z)}{\abs{x - z}}\right)^{2n}f\left(z + \frac{\lambda(z)^2(x - z)}{\abs{x- z}^2}\right)
	\end{equation*}
	and
	\item For every $e\in \bb S^{n - 1}$ there exists $\lambda(e)\in \bb R$ and a set of full measure in $\bb R^n$ on which 
	\begin{equation*}
		f(x)
		= 
		f(x + 2(\lambda - x\cdot e)e)
	\end{equation*}
\end{enumerate}
then there are $c\geq0$, $d>0$ and $x_0\in \bb R^n$ such that 
\begin{equation*}
	f(x) = c\left(d^2 + \abs{x - x_0}^2\right)^{-n}. 
\end{equation*}
\end{theorem}

The following two lemmas will be used. 
\begin{lemma}[\cite{Li2004} Lemma 5.7]
\label{lemma:LiLemma5.7}
Let $n\geq 1$ and $\mu\in \bb R$. If $f:\bb R^n\to \bb R$ satisfies 
\begin{equation*}
	\left(\frac{\lambda}{\abs{z - x}}\right)^\mu f\left(z +\frac{\lambda^2(x- z)}{\abs{x - z}^2}\right)
	\leq
	f(x)
\end{equation*}
for all $\lambda> 0, z\in \bb R^n$ and  $x\in \bb R^n\setminus B_\lambda(z)$ then $f$ is constant. 
\end{lemma}
%
%
%
\begin{proof}[Proof of Theorem \ref{theorem:ExtremalSystemClassification}]
First we claim that if there exists $z_0\in \bdy \bb R_+^n$ such that $\bar\lambda(z_0)< \infty$ then $\bar \lambda(z)< \infty$ for every $z\in \bdy \bb R_+^n$. To see that the claim holds, observe that for any $z\in \bdy \bb R_+^n$ and any $0< \lambda< \bar\lambda(z)$  we have $u\geq u_{z, \lambda}$  a.e. in $\bdy \bb R_+^n\setminus B_\lambda$. Consequently, 
\begin{equation*}
	\liminf_{\abs y\to\infty} \abs{y}^{n - \alpha} u(y)
	\geq
	\liminf_{\abs y\to\infty} \abs{y}^{n - \alpha} u_{z, \lambda}(y)
	= 
	\lambda^{n - \alpha}u(z). 
\end{equation*}
On the other hand, since $\bar\lambda(z_0)< \infty$, by Lemma \ref{lemma:SymmetrySphere} we have
\begin{equation*}
	\liminf_{\abs y\to\infty} \abs{y}^{n - \alpha} u(y)
	= 
	\liminf_{\abs y\to\infty}\abs y^{n - \alpha} u_{z_0, \bar\lambda(z_0)}(y)
	= 
	\bar\lambda(z_0)^{n-\alpha} u(z_0). 
\end{equation*}
Combining the previous two computations gives
\begin{equation*}
	\lambda^{n - \alpha} u(z)
	\leq
	\bar\lambda(z_0)^{n - \alpha} u(z_0)
	\qquad
	\text{ for all }0< \lambda< \bar\lambda(z)
\end{equation*}
so $\bar\lambda(z)< \infty$. The claim is established. \\

Next, we claim that $\bar \lambda(z)<\infty$ for all $z\in \bdy \bb R_+^n$. To see this suppose for sake of obtaining a contradiction that there is $z\in \bdy \bb R_+^n$ such that $\bar\lambda(z) = +\infty$. In this case, the previous claim implies that $\bar\lambda(z) = \infty $ for all $z\in \bdy \bb R_+^n$. Specifically, for every $z\in \bdy \bb R_+^n$ and every $\lambda>0$ we have 
\begin{equation*}
	u(y)
	\geq
	\left(\frac{\lambda}{\abs{y - z}}\right)^{n - \alpha}
	u\left(z + \frac{\lambda^2(y- z)}{\abs{y - z}^2}\right)
\end{equation*}
a.e. in $\bdy \bb R_+^n\setminus B_\lambda(z)$. 
Since $u>0$, applying Lemma \ref{lemma:LiLemma5.7} to $u$ shows that $u = C_0$ for some constant $C_0>0$. This contradicts the assumption $u\in L^{\theta + 1}(\bdy \bb R_+^n)$. The claim is established. \\

By the previous two claims we have $\bar \lambda(z)< \infty$ for all $z\in \bdy \bb R_+^n$. Lemma \ref{lemma:SymmetrySphere} guarantees that both \eqref{eq:uSymmetries} and \eqref{eq:vSymmetries} hold for all $z\in \bdy \bb R_+^n$. In particular $u^{\theta + 1}$ satisfies item (a) of Theorem \ref{theorem:FrankLiebClassification} on $\bb R^{n - 1}$. Next we argue that $u^{\theta+1}$ satisfies the hypothesis (b) of Theorem \ref{theorem:FrankLiebClassification}.  Since the argument is similar to arguments we've already carried out in detail, we only give an outline of the argument. Arguing similarly to the proofs of Lemmas \ref{lemma:KeyEstimatesForMovingPlanes} and \ref{lemma:LambdaBarWellDefined} (with $\sigma = \tau = 0$) we find that for all $e\in \bb S^{n - 2}\subset\bdy \bb R_+^n$ there exists $\lambda(e)\in \bb R$ such that if $\mu\geq \lambda(e)$ then both $u\leq u\circ R_{e, \mu}$ a.e. in $\{y\cdot e\geq \mu\}\cap \bdy \bb R_+^n$ and $v\leq v\circ R_{e, \mu}$ a.e. in $\{x\cdot e\geq \mu\}\cap \bb R_+^n$, where $R_{e,\mu}(x) = x + 2(\mu- x\cdot e)e$ is the reflection of $x$ about the hyperplane $\{x\cdot e = \mu\}$. Define $\lambda^*(e)$ to be the supremum over the collection of all such $\lambda(e)$ and observe that $\lambda^*(e)< \infty$ by the existence $\lambda(-e)$. Arguing similarly to the proof of Lemma \ref{lemma:SoftInequalityAtMaximalPlane} we find both that $u\leq u\circ R_{e, \lambda^*(e)}$ a.e. in $\{y\cdot e\geq \lambda^*(e)\}\cap \bdy \bb R_+^n$ and that $v\leq v\circ R_{e, \lambda^*(e)}$ a.e. in $\{x\cdot e\geq \lambda^*(e)\}\cap \bb R_+^n$. Moreover, exactly one of the following alternatives must hold:
\begin{enumerate}[(a)]	
	\item $u< u\circ R_{e, \lambda^*(e)}$ a.e. in $\{y\cdot e > \lambda^*(e)\}\cap \bdy \bb R_+^n$ and $v< v\circ R_{e, \lambda^*(e)}$ a.e. in $\{x\cdot e> \lambda^*(e)\}\cap \bb R_+^n$, or
	\item $u= u\circ R_{e, \lambda^*(e)}$ a.e. in $\{y\cdot e \geq \lambda^*(e)\}\cap \bdy \bb R_+^n$ and $v= v\circ R_{e, \lambda^*(e)}$ a.e. in $\{x\cdot e\geq \lambda^*(e)\}\cap \bb R_+^n$. 
\end{enumerate}
Arguing similarly to the proof of Lemma \ref{lemma:PlaneLambdaBar=0} we find that if alternative (a) holds then one can contradict the maximality of $\lambda^*(e)$ and thus, alternative (b) must hold. Finally, since $R_{e, \lambda^*(e)}^2 = {\rm id}$, alternative (b) is sufficient to conclude that $u = u\circ R_{e, \lambda^*(e)}$ a.e. on $\bdy \bb R_+^n$ and that $v = v\circ R_{e, \lambda^*(e)}$ a.e. in $\bb R_+^n$. In particular, $u^{\theta + 1}$ satisfies hypothesis (b) of Theorem \ref{theorem:FrankLiebClassification}. Applying Theorem \ref{theorem:FrankLiebClassification} to $u^{\theta + 1}$ guarantees the existence of $c_1>0$, $d>0$ and $y_0\in \bdy \bb R_+^n$ such that \eqref{eq:uClassified} holds. 
\end{proof}
%
\section{Appendix}
\label{section:Appendix}
%
\begin{lemma}
\label{lemma:MappingPropertiesBallExtension}
Let $n\geq 2$. If $\alpha, \beta$ satisfy \eqref{eq:MinimalAlphaBeta} and \eqref{eq:AlphaBetaSubAffine} then $E_B$ maps $L^\infty(\bdy B)$ into $L^{2n/(n - \alpha - 2\beta)}(B)$. 
\end{lemma}
\begin{proof}
It is sufficient to show that $E_B1\in L^{2n/(n - \alpha - 2\beta)}(B)$. We split the proof into cases. \\
{\bf Case 1:} Assume $\alpha + \beta >1$. In this case, for all $(\xi, \zeta)\in B\times \bdy B$ we have 
\begin{equation*}
	H(\xi, \zeta)
	\leq 
	\abs{\xi - \zeta}^{\alpha +\beta - n}
	\leq
	2^{n-\alpha - \beta}\abs{\pi(\xi) - \zeta}^{\alpha +\beta - n}, 
\end{equation*}
where $\pi(\xi) = \xi/\abs\xi$ if $\xi\in B\setminus\{0\}$ and $\pi(0) = 0$. Consequently for all $\xi\in B$
\begin{eqnarray*}
	\int_{\bdy B} H(\xi, \zeta)\; dS_\zeta
	& \leq & 
	C(n, \alpha, \beta)
	\int_{\bdy B}\abs{\pi(\xi) - \zeta}^{\alpha +\beta - n}\; \d S_\zeta
	\\
	& \leq & 
	C(n, \alpha, \beta)
	\int_{\bdy B}\abs{e_n - \zeta}^{\alpha + \beta - n}\; \d S_\zeta
	\\
	& \leq & 
	C(n, \alpha, \beta), 
\end{eqnarray*}
the final estimate holding as $\alpha + \beta >1$. In particular $E_B1\in L^\infty(B)\subset L^{2n/(n - \alpha - 2\beta)}(B)$. \\
{\bf Case 2:} Assume $0< \alpha + \beta \leq 1$. Case 2 will be split into two subcases. \\
{\bf Case 2 (a):} Assume $0< \alpha + \beta \leq 1$ and $\alpha < 1$. The function $\Phi_\alpha:B\to \bb R$ given by 
\begin{equation*}
	\Phi_\alpha(\xi)
	= 
	\int_{\bdy B}\abs{\xi - \zeta}^{\alpha - n}\; \d S_\zeta
\end{equation*}
satisfies $\Phi_\alpha(\xi) = \Phi_\alpha(\abs \xi e_n)$. Moreover, for any $\abs{\xi}< 2$, defining $x^0 = (0', x^0_n) = x^0_ne_n$, where
\begin{equation*}
	x_n^0 = 2\frac{1 - \abs\xi}{1 + \abs \xi}
\end{equation*}
we have $\abs{\xi}e_n = T^{-1}(x^0)$. Using equation \eqref{eq:DistanceBetweenPreimages} and the change of variable $\zeta = T^{-1}(y)$ we get
\begin{equation*}
	\Phi_\alpha(\abs{\xi}e_n)
	= 
	w(x^0)^{\alpha - n}\int_{\bdy \bb R_+^n}\frac{w(y)^{n + \alpha - 2}}{\left( (x_n^0)^2 + \abs y^2\right)^{(n - \alpha)/2}}\; \d y, 
\end{equation*}
where $w$ is as in \eqref{eq:SubcriticalWeight}. Using the change of variable $y\mapsto y/x_n^0$ and equation \eqref{eq:DistanceToBoundaryUnderPreimage} we obtain 
\begin{equation*}
	\left(\frac{1 - \abs\xi^2}{2}\right)^{1 - \alpha}\Phi_\alpha(\abs{\xi}e_n)
	= 
	w(x^0)^{-n - \alpha + 2}\int_{\bdy \bb R_+^n}\frac{w(x_ny)^{n + \alpha - 2}}{\left(1 + \abs y^2\right)^{(n - \alpha)/2}}\; \d y. 
\end{equation*}
The Dominated Convergence guarantees that 
\begin{equation*}
	\lim_{\abs \xi \to 1^-}\left(\frac{1 - \abs\xi^2}{2}\right)^{1 - \alpha}\Phi_\alpha(\abs{\xi}e_n)
	= 
	\int_{\bdy \bb R_+^n}\left(1 + \abs y^2\right)^{(n - \alpha)/2}\; \d y. 
\end{equation*}
Since $\Phi_\alpha\in C^0(B)$ there is $C = C(n, \alpha, \beta)>0$ such that for every $\xi \in B$, 
\begin{equation*}
	\int_{\bdy B}H(\xi, \zeta)\; \d S_\zeta
	= 
	\left(\frac{1 - \abs\xi^2}{2}\right)^\beta \Phi_\alpha (\xi)
	\leq
	C(n, \alpha, \beta)(1 - \abs\xi)^{\alpha + \beta - 1}. 
\end{equation*}
If $\alpha + \beta = 1$ then this gives $\xi\mapsto \int_{\bdy B}H(\xi, \zeta)\; \d S_\zeta \in L^\infty(B)\subset L^{2n/(n - \alpha - 2\beta)}(B)$. If $0< \alpha + \beta < 1$ then \eqref{eq:AlphaBetaSubAffine} guarantees that $0< \frac{2n(1 - \alpha - \beta)}{n - \alpha -2\beta}< 1$ and hence $\xi\mapsto \int_BH(\xi, \zeta)\; \d S_\zeta \in L^{2n/(n -\alpha - 2\beta)}(B)$. \\
{\bf Case 2 (b):} Assume $\alpha = 1$ and $\beta = 0$. Computing similarly to Case 2 (a) we find that 
\begin{equation}
\label{eq:Phi1Decomposition}
	\Phi_1(\xi)
	= 
	\Phi_1(\abs \xi e_n)
	= 
	w(x^0)^{1 - n}\left(I(x^0) + II(x^0)\right), 
\end{equation}
where 
\begin{equation*}
	I(x^0) = \int_{B_2^{n - 1}}\frac{w^{n - 1}(y)}{\abs{x^0  - y}^{n - 1}}\; \d y
\end{equation*}
and 
\begin{equation*}
	II(x^0) = \int_{\bdy\bb R_+^n\setminus B_2^{n - 1}}\frac{w^{n - 1}(y)}{\abs{x^0  - y}^{n - 1}}\; \d y. 
\end{equation*}
For $x_n^0< 1$ 
\begin{equation}
\label{eq:SecondIntegralEstimate}
	II(x^0)
	\leq
	C \int_{\bdy \bb R_+^n\setminus B_2^{n -1}}\frac{w^{n - 1}(y)}{\abs y^{n - 1}}\; \d y
	\leq
	C(n)\int_{\bdy \bb R_+^n\setminus B_2^{n - 1}}\abs y^{-2(n - 1)}\; \d y
	\leq
	C(n). 
\end{equation}
Using the change of variable $y\mapsto y/x_n^0$ and the fact that $\norm{\Grad w}_{L^\infty(\bdy \bb R_+^n)}\leq C(n)$ gives
\begin{equation*}
	I(x^0)
	= 
	\int_{B^{n - 1}(0', 2/x_n^0)}\frac{w(x_n^0y)^{n - 1}}{\left( 1 + \abs y^2\right)^{(n - 1)/2}}\; \d y
	\leq
	\int_{B^{n - 1}(0', 2/x_n^0)}\frac{w(0)^{n - 1} + C(n) x_n^0\abs y}{\left( 1 + \abs y^2\right)^{(n - 1)/2}}\; \d y. 
\end{equation*}
Combining this estimate with the inequalities
\begin{equation*}
	\int_{B^{n - 1}(0', 2/x_n^0)}\left(1 + \abs y^2\right)^{(1- n)/2}\; d y
	\leq
	-C(n)\log x_n^0
\end{equation*}
and
\begin{equation*}
	\int_{B^{n - 1}(0', 2/x_n^0)}\frac{x_n^0\abs y}{\left( 1 + \abs y^2\right)^{(n - 1)/2}}\; \d y
	\leq
	C(n)
\end{equation*}
gives
\begin{equation}
\label{eq:FirstIntegralBoundaryLimsup}
	\limsup_{x_n^0\to 0^+} 
	\frac{-1}{\log x_n^0}I(x^0) 
	\leq
	C(n)w(0)^{n - 1}. 
\end{equation}
Finally, since $\log(x_n^0)/\log(1 - \abs\xi)\to 1$ as $\abs\xi \to 1$ using \eqref{eq:SecondIntegralEstimate} and \eqref{eq:FirstIntegralBoundaryLimsup} in \eqref{eq:Phi1Decomposition} gives
\begin{equation*}
	\limsup_{\abs \xi\to 1^-}\frac{-\Phi_1(\xi)}{\log(1 - \abs \xi)}
	\leq
	C(n). 
\end{equation*}
Consequently there is $C(n)>0$ such that for all $\xi\in B$, 
\begin{equation*}
	\int_{\bdy B}H(\xi, \zeta)\; \d S_\zeta
	= 
	\Phi_1(\xi)
	\leq
	C(n)\abs{\log(1 - \abs \xi)}
	\in 
	L^{2n/(n - \alpha - 2\beta)}(B). 
\end{equation*}
\end{proof}
\bibliographystyle{plain}

\end{document}